\documentclass[11pt]{article}
\usepackage{amsmath,amsfonts,amssymb,amsthm,amscd,xypic, hyperref}
\usepackage{enumitem}

\title{Analytic spectral flow formula for unitaries and Levinson's theorem}
\author{Angus Alexander, Alan Carey, Galina Levitina, Adam Rennie\thanks{email:\,\texttt{angusa@uow.edu.au,\,renniea@uow.edu.au,\,galina.levitina@anu.edu.au,\,alan.carey@anu.edu.au }}}

\advance\topmargin by -\headheight
\advance\topmargin by -\headsep
\evensidemargin=\oddsidemargin
\numberwithin{equation}{section} 

\theoremstyle{plain} 
\newtheorem{thm}{Theorem}[section]
\newtheorem{lemma}[thm]{Lemma}
\newtheorem{prop}[thm]{Proposition}
\newtheorem{corl}[thm]{Corollary}

\theoremstyle{definition} 
\newtheorem{defn}[thm]{Definition}

\newtheorem{rmk}[thm]{Remark} 

\DeclareMathOperator{\Dom}{Dom}   
\DeclareMathOperator{\End}{End}   
\DeclareMathOperator{\Id}{Id}     
\DeclareMathOperator{\Log}{Log} 
\DeclareMathOperator{\Tr}{Tr}     
\DeclareMathOperator{\tr}{tr}     
\DeclareMathOperator{\Det}{Det}     
\DeclareMathOperator{\spf}{sf}     
\DeclareMathOperator{\slim}{\textup{s-lim}}

\newcommand{\eps}{\varepsilon} 
\newcommand{\e}{\mathrm{e}}

\newcommand{\B}{\mathcal{B}}  
\newcommand{\C}{\mathbb{C}}   
\renewcommand{\H}{\mathcal{H}}  
\newcommand{\K}{\mathcal{K}} 
\renewcommand{\L}{\mathcal{L}} 
\newcommand{\N}{\mathbb{N}}   
\newcommand{\R}{\mathbb{R}}   
\newcommand{\Sf}{\mathbb{S}}  
\newcommand{\Z}{\mathbb{Z}}   

\newcommand{\ox}{\otimes}     

\renewcommand{\d}{\mathrm{d}} 

\newcommand{\pairing}[2]{(#1\mathbin{|}#2)} 
\newcommand{\norm}[1]{\left|\left|#1 \right| \right|}

\begin{document}

\maketitle

\begin{abstract}

We prove an integral formula for the spectral flow of differentiable loops of unitaries of the form ${\rm Id}+$Schatten. Our formula is in terms of a  regularised winding number, expressed in terms of exact differential forms, and we show how the formula extends to non-closed paths. 
Applying these ideas to the scattering operator of Schr\"{o}dinger scattering systems yields explicit formulae for the number of bound states, possibly modified by the presence of resonances, of the system in terms of the potential. We finish by briefly considering the paths of unbounded operators obtained from unitary loops via the Cayley transform. These include cases of moving domain as well as paths with non-constant Hilbert space.
\end{abstract}


\parindent=0.0in
\parskip=4pt

\section{Introduction}

Starting from the invention by Lusztig, the notion of spectral flow played a significant role in geometry as demonstrated by  Atiyah, Patodi and Singer \cite{aps75, aps76}. Initially defined in topological   terms as an intersection number, later an analytical definition was given by Phillips \cite{Phi96,Phi97}. Exploiting an essential contribution of Getzler \cite{Getzler}, the papers \cite{CP1,CPotSuk}  established  spectral flow formulae as integrals of one-forms on affine subspaces of the Banach manifold of self-adjoint Fredholms.
In this paper we aim to establish such an integral formula for spectral flow of a path of unitaries which are Schatten $p$-class perturbations of 
 the identity, $p\geq 1$. We are motivated by scattering theory where a path of scattring matrices is a Schatten norm continuous  path $S(\lambda)$ of $d$-Schatten unitaries depending on the spectral variable.
 
 The spectral flow along a path of unitaries is well-established. We utilise the definition and several features of known approaches to spectral flow for unitaries appearing in \cite{ADT,BBF,BLP,DSBW,KirLes}.  Other approaches, such as Pushnitski's \cite{Pu}, necessarily agree with ours but have quite different definitions. 
 
For the spectral flow of a differentiable loop $U_\bullet$ of unitaries which are trace-class perturbations of the identity, the following integral formula holds \cite{BLP,DSBW,KirLes}: 
\begin{equation}\label{eq:sf_trace_class_intro}
\textup{sf}(U_\bullet)=\frac{1}{2\pi i}\int_0^1\Tr(U_t^*\dot{U}_t)\,\d t,
\end{equation}
showing that the spectral flow can be computed via the integral of an exact one form. A refinement of the formula \eqref{eq:sf_trace_class_intro} for Schatten class unitaries was proposed in \cite{kellendonk12}, and the formula suggested there is what we ultimately obtain, and now describe.

In this paper (see Theorem \ref{thm:cayley-spec-flow-formula} below) we show that if $U_\bullet$ is a Schatten differentiable loop of unitaries, such that $U_t-\Id$ is in the Schatten $p$-class for some $p\geq 1$, then
 for $r\geq (p-1)/2$ or $n\geq p-1$ we have 
\begin{align}\label{eq:sf_Schatten_class_intro}
\begin{split}
\textup{sf}(U_\bullet)&=\frac{-i}{2^{2r+1}}\frac{\Gamma(r+1)}{\sqrt{\pi}\Gamma(r+1/2)}\int_0^1\Tr\Big(U_t^*\frac{\d}{\d t}(U_t)|U_t-{\rm Id}|^{2r}\Big)\,\d t\\
&=(-1)^n\frac{1}{2\pi i}\int_0^1\Tr\Big(U_t^*\frac{\d}{\d t}(U_t)(U_t-{\rm Id})^{n}\Big)\,\d t.
\end{split}
\end{align}
Here the integrand in both integrals is an exact differential form on the Banach-Lie group of unitaries of the form ${\rm Id}+$Schatten. Notably, Equation \eqref{eq:sf_Schatten_class_intro} holds for any loop of unitaries indexed by any interval (bounded or not) on the real line. For example, if  $U_\bullet$ is a differentiable loop of unitaries indexed by $[0,\infty)$, such that $U_s-\Id$ is in the Schatten $p$-class for some $p\geq 1$ for any $s\in [0,\infty)$, then we have 
\begin{align}\label{eq:sf_Schatten_class_intro_2}
\begin{split}
\textup{sf}(U_\bullet)&=\frac{-i}{2^{2r+1}}\frac{\Gamma(r+1)}{\sqrt{\pi}\Gamma(r+1/2)}\int_0^\infty\Tr\Big(U_s^*\frac{\d}{\d s}(U_s)|U_s-{\rm Id}|^{2r}\Big)\,\d s\\
&=(-1)^n\frac{1}{2\pi i}\int_0^\infty\Tr\Big(U_s^*\frac{\d}{\d s}(U_s)(U_s-{\rm Id})^{n}\Big)\,\d s,
\end{split}
\end{align}
provided that the integrand defines an integrable function on $[0,\infty)$.

We also present an alternative formula for the spectral flow of unitaries in terms of regularised determinants in Theorem \ref{thm:spec-flow-formula-det}, again first suggested in \cite{kellendonk12}. This states that for $p \in \N$ and $U_\bullet: [0,1] \to \mathcal{U}_p$ a differentiable loop of $p$-Schatten  unitaries we have
\begin{align*}
\textup{sf}(U_\bullet)&=\frac{1}{2\pi i}\int_0^1\frac{\frac{\d}{\d t}\textup{Det}_p(U_t)}{\textup{Det}_p(U_t)}\,\d t.\end{align*}
We extend these spectral flow formulae to non-closed paths. The usual procedure, which for self-adjoint Fredholm operators yields eta invariants of the end-points, depends on the choice of paths from the identity to the endpoints. While this dependence can not in general be removed, we use the known cut-locus of the unitary groups to identify a convention when the end-points are finite rank perturbations of the identity. This convention is natural and produces known answers in examples, as we demonstrate in the following section.

 By connecting the endpoints of a non-closed path to the identity via geodesic paths, we prove an analytic formula for the spectral flow along the path in terms of regularised winding number formula and end-point correction terms (see Section \ref{subsec:paths}).

In Section \ref{sec:Lev-det}
we apply our spectral flow results to Schr\"{o}dinger scattering systems. We formulate the celebrated theorem of Levinson in terms of spectral flow along the path of unitaries given by the scattering matrix.  Levinson's theorem relates the number of bound states of the  Schr\"{o}dinger operator $H=-\Delta+V$ to the scattering matrix $S(\lambda)$ of the pair $(-\Delta,-\Delta+V)$.
 
There have been many interpretations of Levinson's theorem, see e.g.  \cite{alexander24, ANRR, AR23-4D, bolle88, bolle77, levinson49}. Notably, for each energy level $\lambda\in [0,\infty)$,  the scattering matrix $S(\lambda)$ is a unitary operator on an appropriate Hilbert space,  so that $S(\bullet)$ is a path of unitaries indexed by $[0,\infty)$. Furthermore, under suitable assumptions on $V$, $S(\lambda)$  is a trace-class perturbation of the identity. However, the trace of $S^*(\lambda)S'(\lambda)$ is usually not integrable on $[0,\infty)$, making it impossible to apply the spectral flow formula \eqref{eq:sf_trace_class_intro}. However, as we show in Section \ref{sec:Lev-det} the spectral flow formulae \eqref{eq:sf_Schatten_class_intro_2}, may be applied in this situation. 

For nice potentials $V$, we show in all generality (see Theorem \ref{thm:sf_Levinson2} below) that 
\begin{align}
\textup{sf}(S(\bullet)) &=  \frac{1}{2\pi i} \int_0^\infty \left( \textup{Tr}(S(\lambda)^*S'(\lambda)) - p_n(\lambda) \right) \, \d \lambda - \frac{1}{2\pi i} P_n(0)= -N-N_{res} ,
\label{eq:intro-lev}
\end{align}
where $N$ is the number of eigenvalues of $H = -\Delta+V$ and $N_{res}$ is an integer correction arising in dimensions 2 and 4 only in the presence of zero energy resonances. In particular, \eqref{eq:intro-lev} holds in all dimensions, whether resonances are present or not.
 The polynomials  $P_n(\lambda),$ and $p_n(\lambda)$ are explicitly defined in term of the potential $V$. 
 

Given a loop of unitaries $U_\bullet$ on a Hilbert space $\H$,
we can apply the Cayley transform to this loop of unitaries, to obtain a path of unbounded operators. These paths, while well-behaved in many ways, have the disturbing feature that the operators at each point may fail to be densely-defined on $\H$. By completing the domains of definition, we obtain paths of self-adjoint operators on moving Hilbert spaces, where the standard definition of the spectral flow is not applicable.  Nevertheless, we can define a sensible notion of spectral flow for such paths, via the unitary spectral flow.

Our results for paths of operators encompass known cases of gap continuous varying boundary conditions, \cite{BLP,Wahl07}. To understand better what the abstract case of paths of self-adjoint operators on moving Hilbert spaces says, we formulate our results using Kasparov's bivariant K-theory. 

Kasparov's stabilisation theorem \cite{Ka1.5} shows that any Hilbert module $X$ over $C_0(\R)$ consists of a continuous field of Hilbert spaces $\{\H_t\}$ over $\R$, where each $\H_t$ is a closed subspace of a fixed Hilbert space $\H$. Given an unbounded self-adjoint regular operator $T$ on $X$, we characterise the continuity of the family $\{T_t\}$ on the continuous field $\{\H_t\}$, and when $\{T_t\}$ defines a Kasparov module. Our characterisation, found in Proposition \ref{prop:gap-cts-extension} and Theorem \ref{prop:KK-paths}, is in terms of continuous paths in a Banach manifold arising from the Banach manifold $\mathcal{U}_\infty(\H)$ of unitaries of the form ${\rm Id}+$compact via the Cayley transform.



The discussion of exact forms and our formulae for spectral flow are in Section \ref{sec:sf-su}. The application to scattering theory is in Section \ref{sec:Lev-det}, and the paths of unbounded operators are discussed in Section \ref{sec:fctns}.

{\bf Acknowledgements} All authors were supported by the ARC Discovery grant DP220101196, and thank Serge Richard for sharing his insights on Levinson's theorem. A.R. thanks the Institute for Advanced Research, Nagoya University, 
for support and hospitality during the production of this work.
%
%
%
%
%
%
%
%
%

\section{Spectral flow for Schatten unitaries}
\label{sec:sf-su}

Throughout the paper we denote by $\H$ a separable Hilbert space and the algebra of all bounded (respectively, compact) linear operators on $\H$ is denoted by $\B(\H)$ (respectively, by $\K(\H)$). The standard trace on the algebra $\mathcal{B}(\H)$ is denoted by $\Tr$.  For $p\geq 1$  let $\L^p(\H)$ be the Schatten $p$-class of compact operators $T$ on $\H$ satisfying $\Tr(|T|^p)<\infty$.

Our aim is to study the spectral flow of moderately smooth paths $(U_t)$ of $\Id+$Schatten unitaries through the point $-1$ in the spectrum. The issue has been addressed before, and a good introduction is provided in \cite{DSBW}.
We start with the definition of spectral flow for a path of self-adjoint (bounded) Fredholm operators due to Phillips \cite{Phi96,Phi97}.

Let  $\pi: \B(\H) \to \B(\H) \slash \K(\H)$ denote the quotient onto the Calkin algebra and let  $\chi = \chi_{[0,\infty)}$ be the characteristic function of the interval $[0,\infty)$. 

Assume that $(T_t), t\in[0,1],$ is  any norm continuous path of bounded self-adjoint Fredholm operators in $\mathcal{B}(\H)$. Then $\pi(\chi(T_t)) = \chi(\pi(T_t))$. Since the spectrum of the $\pi(T_t)$ are bounded away from zero, the path $\chi(\pi(T_t))$ is continuous. By compactness (and \cite[Lemma 4.1]{BCPRSW}), we can choose a partition $0 = t_0 < t_1 < \cdots < t_k  = 1$ such that 
\begin{align*}
\norm{\pi(\chi(T_{t})) - \pi(\chi(T_s))} &< \frac12
\end{align*} 
for all $t, s \in [t_{i-1}, t_{i}]$ and $0 \leq i \leq k$. Defining the projection $P_i = \chi(T_{t_i})$ we find that $P_{i-1}P_i: P_i \H \to P_{i-1} \H$ is Fredholm. The following is the definition due to Phillips \cite{Phi96, Phi97}.
\begin{defn} For $t \in [0,1]$ let $(T_t)$ be any norm continuous path of bounded self-adjoint Fredholm operators in $\B(\H)$.  For a partition $0 = t_0 < t_1 < \cdots < t_k = 1$ of the interval $[0,1]$ define the operators $P_i = \chi(T_{t_i})$. Then we define the \emph{spectral flow} of the path $(T_t)$ by
\begin{align*}
\textup{sf}(T_t) &:= \sum_{i=1}^k{\textup{Ind}(P_{i-1}P_i)}.
\end{align*}
\end{defn}
This definition of spectral flow is independent of the choice of partition \cite{LU, Phi96, Phi97} and agrees with the topological definition used in \cite{aps75, aps76} when both apply.

If $D_t$ is a path of densely-defined self-adjoint operators with
$t\mapsto F_t:=D_t(1+D_t^2)^{-1/2}$ operator norm continuous, then we say that the path $D_t$ is Riesz continuous. If the path of graph projections $P_{D_t}$ are norm continuous, we say that the path $D_t$ is gap continuous.
Gap continuity is equivalent to norm resolvent continuity, \cite[Theorem 1.1]{BLP}.

Spectral flow for Riesz continuous paths of unbounded self-adjoint Fredholm operators $D_t$ is defined as the spectral flow of the path of bounded transforms $F_t$ using Phillips' definition above.

In \cite[Section 2.1]{BLP} it is also shown that, via the Cayley transform, the spectral flow of a gap continuous path $\{D_t\}$ of unbounded self-adjoint Fredholm operators can be defined as the winding number of the associated path of unitaries $t\mapsto (D_t+i)(D_t-i)^{-1}$. For a gap continuous path, \cite[Definition 2.12]{BLP} shows that Phillips' definition still applies.

\subsection{Spectral flow of unitaries}
Now we move on to an analytical definition of spectral flow for a path of unitaries. We follow the discussions of \cite{BBF,BLP,DSBW,KirLes} where the spectral flow of unitaries and the relation to spectral flow of closed unbounded Fredholm operators is presented. Pushnitski's treatment, \cite{Pu}, also deals  with the Schatten class unitaries that we consider, but having different aims develops a quite different approach. In particular, Pushnitski deals with the discontinuity of integer-valued dimension-type functions on the circle by building a space of such functions, together with a natural covering space. By defining spectral flow in terms of lifts of maps to the covering space, Pushnitski obtains homotopy invariance. 

Our approach will be closer to the ``naive'' definition that counts the net number of eigenvalues passing through $-1$ as we traverse a loop of unitaries. This approach can be made rigorous by a Phillips-style definition presented in \cite{BBF} and \cite[Proposition 2.1]{BLP}, which is our starting point. We refer to \cite[Section 2]{BLP} and \cite{KirLes} for more discussion and additional references.

Our basic definitions will make sense for the ``Fredholm unitaries'', see \cite[Corollary 1.8]{BLP} and \cite[Section 6.1]{KirLes}, namely those Hilbert space unitaries for which $-1$ is not contained in the essential spectrum. The Fredholm unitaries do not form a group, and its elements do not have sufficiently nice behaviour for us to define our differential forms.
It is shown in \cite[Lemma 6.1]{KirLes} and \cite[Proposition 3.7.2]{DSBW} that the unitaries of the form ${\rm Id}+K$ with $K$ compact are a deformation retract of the Fredholm unitaries. 

To define one-forms on a Banach manifold of unitaries, we will require $p$-Schatten class compacts, which we introduce next. We note that these  are again homotopy equivalent to the unitaries of the form ${\rm Id}+K$, \cite{dlH}.

\begin{defn}Let $p\geq 1$. We say a unitary $U$ on $\H$ is \emph{$p$-Schatten} if $U=\Id+T$ for some  $T\in \L^p(\H)$. We say that a path $[0,1]\ni t\mapsto U_t$ of $p$-Schatten unitaries is piecewise \emph{differentiable} (or  piecewise $C^1$) if the path is continuously differentiable at all but finitely many points $t\in(0,1)$ in the Schatten norm $\|\cdot\|_p$ and is continuous on the whole interval $[0,1]$ in the Schatten norm $\|\cdot\|_p$. We denote by $\mathcal{U}_p(\H)$ the set of Schatten unitaries on $\H$ with the topology coming from the Schatten norm $\Vert\cdot\Vert_p$.
\end{defn}

For the remainder of this section we fix $p\geq 1$.

It is well-known that $\mathcal{U}_p(\H)$
is a Banach-Lie group with Lie algebra the skew-adjoint subspace of $\L^p(\H)$, \cite{dlH}. Similarly, the space  $\mathcal{U}_\infty(\H)$ of unitaries of the form ${\rm Id}+X,\ X\in \K(\H)$ with the norm topology is a Banach-Lie group.

%

Let $U_t : [0, 1]\to \mathcal{U}_p$ be a continuous path. By \cite[Proposition 2.1]{BLP} there exists a partition $\{0 = t_0 < t_1 <\cdots < t_n = 1\}$ of the interval and positive real
numbers $0 < \epsilon_j < \pi$, $j = 1, \dots , n$, such that $\ker(U_t- e^{i(\pi\pm\epsilon_j )}) = \{0\}$ for $t_{j-1} \leq t \leq t_j$.
Introduce the notation 
\[
k(t, \epsilon_j ) :=\sum_{0\leq\theta<\epsilon_j}
\dim \ker(U_t- e^{i(\pi+\theta)}).
\]
The sum defining $k(t, \epsilon_j )$ is necessarily finite.
\begin{defn}
\label{defn:BLP-SF}
Let $U_t : [0, 1]\to \mathcal{U}_p$ be a continuous path. For a finite partition $\{0 = t_0 < t_1 <\cdots < t_n = 1\}$ of the interval and positive real
numbers $0 < \epsilon_j < \pi$, $j = 1, \dots , n$, such that $\ker(U_t- e^{i(\pi\pm\epsilon_j )}) = \{0\}$ for $t_{j-1} \leq t \leq t_j$, define the \emph{spectral flow} for $U_\bullet$ by setting 
\[
\textup{sf}(U_\bullet)=\textup{wind}(U_\bullet ) =
\sum^n_{j=1}
k(t_j , \epsilon_j )- k(t_{j-1}, \epsilon_j ),
\]
\end{defn}
As shown in \cite[Proposition 2.1]{BLP} this definition is independent of the choice of the partition of the interval and of the choice of the $\epsilon_j$.
In \cite{BLP,DSBW,KirLes}, they recall the formula for the winding number on $\mathcal{U}_1(\H)$. For a differentiable loop $U_\bullet\in\mathcal{U}_1(\H)$ based at ${\rm Id}_\H$, we have
\[
\textup{wind}(U_\bullet)=\frac{1}{2\pi i}\int_0^1\Tr(U_t^*\dot{U}_t)\,\d t\in\Z.
\]
Our aim in this section is to extend this formula to all $\mathcal{U}_p(\H)$ with $1\leq p<\infty$.

\subsection{Exact one-forms}
\label{subsec:exact-forms}

 To prove the integral formula \eqref{eq:sf_Schatten_class_intro}, we firstly introduce two different one-forms and show that they are exact. In the following, we fix the Hilbert space $\H$ and write $\mathcal{U}_p=\mathcal{U}_p(\H)$ and similarly for Schatten ideals.

The tangent space $T_{\rm Id}\mathcal{U}_p$ to the identity in $\mathcal{U}_p$ is the skew-adjoint elements of $\L^p$. Given two tangent vectors $X,Y\in  T_{\rm Id}\mathcal{U}_p$ the Lie bracket is given by  $(X,Y)\mapsto [X,Y]$ where $[X,Y]=XY-YX$ is the standard commutator. 
The exponential map  $\exp: T_{\rm Id}\mathcal{U}_p\to \mathcal{U}_p, X\mapsto \exp(X)$ from the tangent space $ T_{\rm Id}\mathcal{U}_p$,  is onto and a local diffeomorphism, \cite[Proposition II.15]{dlH}.

The exponential map takes tangent vectors $X\in T_{\rm Id}\mathcal{U}_p$ to  curves through an arbitrary point $U\in \mathcal{U}_p$ with tangent vector at $U$ a translation of $X$ to the point  $U$,  so
\[
\exp(sX)\cdot U=Ue^{sX},\quad U\in \mathcal{U}_p,\quad X=-X^*\in \L^p.
\]

Now we define the two one-forms on
 the 
tangent space to  $\mathcal{U}_p$ at the point  $U\in\mathcal{U}_p.$

\begin{defn}
\label{defn:one-forms}
For $x\geq 0$ let 
\[
C_x=\frac{\Gamma(x+1)}{\sqrt{\pi}\Gamma(x+1/2)}.
\]
Let $n$ be an integer $n\geq p-1$ and let $r$ be real such that  $r\geq (p-1)/2$. Define one-forms on  
tangent vectors $X\in T_U\mathcal{U}_p$ at the point  $U\in\mathcal{U}_p$ by 
\[
\alpha_{U,n}(X):=
(-1)^n\frac{1}{2\pi i}\Tr(X(U-\Id)^n)
\]
and 
\[
\beta_{U,r}(X):=
-iC_r\Big(\frac{1}{2}\Big)^{2r+1}\Tr(X|U-\Id|^{2r}).
\]
\end{defn}

\begin{rmk}
We note that since $U\in \mathcal{U}_p$, the H\"older inequality implies that $(U-\Id)^n\in \L^{\frac{p}{n}}$ and $|U-\Id|^{2r}\in \L^{\frac{p}{2r}}$. Since $X\in \L^p$ and $n\in\mathbb{N}, r\in\mathbb{R}$ are chosen such that $n\geq p-1$ and $r\geq (p-1)/2$, another application of the H\"older inequality implies that both operators under the trace in Definition \ref{defn:one-forms} are trace-class operators, so that both one-forms are well-defined. In what follows, we will repeatedly use this argument without further reference. 
\end{rmk}

In the following two lemmas we show that the forms $\alpha_{U,n}$  and $\beta_{U,n}$ are  exact.  We start with the one-form $\alpha_{U,n}$.

\begin{lemma}
\label{lem:exctness}
For any $U\in \mathcal{U}_p$ and $n\in\mathbb{N}$ such that $n\geq p-1$ the one-form $\alpha_{U,n}$  is  exact.
\end{lemma}

\begin{proof}

Let $U_1\in \mathcal{U}_p$ and and let $Y=-Y^*\in\L^p$ be such that $U_1=e^Y$. Setting $U_t= e^{tY}$. $t\in[0,1]$, we define the function $\Theta:\mathcal{U}_p\to \mathbb{C}$ by setting
\begin{equation}
\Theta(U_1)=(-1)^n\frac{1}{2\pi i}\int_0^1\Tr(Y(U_t-\Id)^{n})\,\d t.
\label{eq:they-called-him-Theta}
\end{equation}
By the H\"older inequality, the function $t\mapsto Y(U_t-\Id)^n$, $t\in[0,1]$, is a continuous  trace-class valued function, and so the integral above is well-defined.

We claim that $d\Theta_{U_1}(X)=\alpha_{U_1,n}(X)$ for any $X\in T_{U_1}\mathcal{U}_p$. 
We firstly note that since $Y=U_t^*\frac{\d U_t}{\d t}$, we can rewrite the definition of $\Theta$ as follows:
\[
\Theta(U_1)=(-1)^n\frac{1}{2\pi i}\int_0^1\Tr\left(U_t^*\frac{\d U_t}{\d t}(U_t-\Id)^{n}\right)\,\d t.
\]

To compute differentials at the point $U_1\in\mathcal{U}_p$, we first compute that 
\[
\Theta(U_1e^{sX})=(-1)^n\frac{1}{2\pi i}\int_0^1\Tr\Big(e^{-(t-1)Y}e^{-sX}U_1^*\frac{\d (U_1e^{sX}e^{(t-1)Y})}{\d t}(U_1e^{sX}e^{(t-1)Y}-\Id)^{n}\Big)\,\d t.
\]
Noting that 
\[
\frac{\d}{\d s}\Big|_{s=0}U_1e^{sX}e^{(t-1)Y}=U_1Xe^{(t-1)Y}=U_1Xe^{-Y}U_0^*U_0e^{tY}=U_1XU_1^*U_t
\]
and
\[
\frac{\d}{\d s}\Big|_{s=0}e^{-(t-1)Y}e^{-sX}U_1^*=-U_t^*U_1XU_1^*
\]
we infer that 
\[
\frac{\d}{\d s}\Big|_{s=0}e^{-(t-1)Y}e^{-sX}U_1^*\frac{\d(U_1e^{sX}e^{(t-1)Y})}{\d t}
=U_t^*U_1XU_1^*\frac{\d U_t}{\d t}-U_t^*U_1XU_1^*\frac{\d U_t}{\d t}=0.
\]

Since, in addition, we have that 
$$e^{-(t-1)Y}e^{-sX}U_1^*\frac{\d (U_1e^{sX}e^{(t-1)Y})}{\d t}\Big|_{s=0}=Y,$$
we obtain that 
$$
(-1)^n2\pi i\frac{\d}{\d s}\Big|_{s=0}\Theta(U_1e^{sX})
=\int_0^1\Tr\Big(Y\frac{\d}{\d s}\Big|_{s=0}(U_1e^{sX}e^{(t-1)Y}-\Id)^{n}\Big)\,\d t.$$

Writing
\begin{align*}
\frac{\d}{\d s}\Big|_{s=0}&(U_1e^{sX}e^{(t-1)Y}-1)^{n})\\
&=\sum_{k=0}^{n-1} (U_1e^{(t-1)Y}-\Id)^k \Big(\frac{\d}{\d s}\Big|_{s=0}(U_1e^{sX}e^{(t-1)Y}-\Id)\Big)(U_1e^{(t-1)Y}-\Id)^{n-k-1}\\
&=\sum_{k=0}^{n-1}(U_t-\Id)^{k}U_1XU_1^*U_t(U_t-\Id)^{n-k-1},
\end{align*}
we find that 
$$(-1)^n2\pi i\frac{\d}{\d s}\Big|_{s=0}\Theta(U_1e^{sX})=\int_0^1\Tr\Big(Y\sum_{k=0}^{n-1}(U_t-\Id)^{k}U_1XU_1^*U_t(U_t-\Id)^{n-k-1}\Big)\,\d t.$$
Using repeatedly the cyclicity of the trace and unitarity of $U_t$ we conclude that 
\begin{align*}
(-1)^n2\pi i\frac{\d}{\d s}\Big|_{s=0}\Theta(U_1e^{sX})
&=\int_0^1\Tr\Big(U_1XU_1^*\sum_{k=0}^{n-1}(U_t-\Id)^{k}U_tY(U_t-\Id)^{n-k-1}\Big)\,\d t\\
&=\int_0^1\Tr\Big(U_1XU_1^*\frac{\d}{\d t}\big(U_t-\Id)^{n}\big)\Big)\,\d t\\
&=\Tr(U_1XU_1^*(U_1-\Id)^{n})
-\Tr(U_1XU_1^*(U_0-\Id)^{n}))\\
&=\Tr(X(U_1-\Id)^{n})\\
&=(-1)^n2\pi i\,\alpha_{U_1,n}(X),
\end{align*}
proving the claim that $d\Theta_{U_1}(X)=\alpha_{U_1,n}(X)$, and thus showing that the one-form $\alpha_{U,n}$ is exact. 
%
\end{proof}

Next we examine the one-forms $\beta_{U,r}$ with $r$ real. The proof of exactness for these forms requires the spectral theorem to handle non-integer powers.

\begin{lemma}
\label{lem:exactness}
For any $U\in \mathcal{U}_p$ and $r\in\mathbb{R}$ such that $r\geq (p-1)/2$, the one-form $\beta_{U,r}$ on $\mathcal{U}_p$ is exact.
\end{lemma}
\begin{proof}

Let $n$ be the integer part of $r$ so $r=n+x$ with $0\leq x<1$.  For $U_t=e^{tY}$, $t\in[0,1]$, define $\Xi:\mathcal{U}_p\to\C$ at the point $U_1$ by
\begin{equation}
\Xi(U_1)=\int_0^1\Tr(Y|U_t-\Id|^{2r})\,\d t=\int_0^1\Tr(Y|U_t-\Id|^{2n}|U_t-\Id|^{2x})\,\d t.
\label{eq:they-called-him-Xi}
\end{equation}
By the H\"older inequality, the function $t\mapsto Y|U_t-\Id|^{2r}$, $t\in[0,1]$, is a continuous  trace-class valued function, and so the integral above is well-defined.

Now, for $x\neq 0$, by the integral formula (see e.g. \cite[1.3.7]{ped})
$$T^{-\alpha}=\frac{\sin(\alpha\pi)}{\pi}\int_0^\infty\lambda^{-\alpha}(1+\lambda T)^{-1}T\,\d\lambda, \quad 0<\alpha<1, 0\leq T\in \B(\H),$$
 we can write 
\begin{equation}
|U_t-\Id|^{2x}=\frac{\sin(x\pi)}{\pi}\int_0^\infty\lambda^{-x}(1+\lambda|U_t-\Id|^2)^{-1}|U_t-\Id|^2\,\d\lambda.
\label{eq:ped}
\end{equation}
As the integral is absolutely convergent in operator norm with both the integrand and the value of the integral $|U_t-\Id|^{2x}\in\L^{p/2x}$ and $Y|U_t-\Id|^{2n}|U_t-\Id|^{2x}\in\L^1$, we may interchange the trace and integral to find that
\[
\Xi(U_1)=\frac{\sin(x\pi)}{\pi}\int_0^1\int_0^\infty\lambda^{-x}\Tr(Y|U_t-\Id|^{2n}(1+\lambda|U_t-\Id|^2)^{-1}|U_t-\Id|^2)\,\d\lambda\,\d t.
\]
We now compute the exterior derivative of $\Xi$ at $U_1$. As before, for $X\in T_{U_1}\mathcal{U}_p$ we have
\[
\d \Xi_{U_1}(X)=\frac{\d}{\d s}\Big|_{s=0}\Xi(U_1e^{sX}).
\]
We recall that
\begin{align*}
\frac{\d}{\d s}\Big|_{s=0}|U_1e^{sX}e^{(t-1)Y}-\Id|^2&=\frac{\d}{\d s}\Big|_{s=0}(2-U_1e^{sX}e^{(t-1)Y}-e^{-(t-1)Y}e^{-sX}U_1^*)\\
&=-(U_1Xe^{(t-1)Y}-e^{-(t-1)Y}XU_1^*)=-(U_1XU_1^*U_t-U_t^*U_1XU_1^*),
\end{align*}
and inspection shows that the derivative lies in $\L^{p}$.
The remaining elements needed to compute the differential are then
\begin{align*}
&\frac{\d}{\d s}\Big|_{s=0}|U_1e^{sX}e^{(t-1)Y}-\Id|^{2n}\\
&\qquad=-\sum_{k=0}^{n-1}|U_t-\Id|^{2k}(U_1XU_1^*U_t-U_t^*U_1XU_1^*)|U_t-\Id|^{2(n-k-1)}
\end{align*}
and 
\begin{align*}
&\frac{\d}{\d s}\Big|_{s=0}(1+\lambda|U_1e^{sX}e^{(t-1)Y}-\Id|^2)^{-1}\\
&\quad=\lim_{s\to 0}\frac{1}{s}(1+\lambda|U_1e^{sX}e^{(t-1)Y}-\Id|^2)^{-1}\Big(\lambda|U_t-\Id|^2-\lambda|U_1e^{sX}e^{(t-1)Y}-\Id|^2\Big)(1+\lambda|U_t-\Id|^2)^{-1}\\
&\quad=\lambda(1+\lambda|U_t-\Id|^2)^{-1}\big(U_1XU_1^*U_t-U_t^*U_1XU_1^*\big)(1+\lambda|U_t-\Id|^2)^{-1}.
\end{align*}
Thus, we have that
\begin{align*}
&\d \Xi_{U_1}(X)=\frac{\d}{\d s}\Big|_{s=0}\Xi(U_1e^{sX})\\
&=-\frac{\sin(x\pi)}{\pi}\int_0^1\int_0^\infty\lambda^{-x}\Tr\Bigg(Y|U_t-\Id|^{2n}(1+\lambda|U_t-\Id|^2)^{-1}(U_1XU_1^*U_t-U_t^*U_1XU_1^*))\\
&-Y\sum_{k=0}^{n-1}|U_t-\Id|^{2k}(U_1XU_1^*U_t-U_t^*U_1XU_1^*)|U_t-\Id|^{2(n-k-1)}(1+\lambda|U_t-\Id|^2)^{-1}|U_t-\Id|^2)\\
&+Y|U_t-\Id|^{2n}\lambda(1+\lambda|U_t-\Id|^2)^{-1}\big(U_1XU_1^*U_t-U_t^*U_1XU_1^*\big)(1+\lambda|U_t-\Id|^2)^{-1}|U_t-\Id|^2)\Bigg)\,\d\lambda\,\d t.
%
\end{align*}
As $U_t$ is normal and $YU_t=U_tY$, we can rearrange these expressions using cyclicity of the trace to find
\begin{align*}
&\d \Xi_{U_1}(X)
=-\int_0^1\Tr(Y|U_t-\Id|^{2n+2x-2}(U_1XU_1^*U_t-U_t^*U_1XU_1^*))\,\d t\\
&-n\int_0^1\Tr(Y|U_t-\Id|^{2n+2x-2}(U_1XU_1^*U_t-U_t^*U_1XU_1^*)\,\d t\\
&+\frac{\sin(x\pi)}{\pi}\int_0^1\int_0^\infty\lambda^{-x+1}\Tr(Y|U_t-\Id|^{2n+2}(1+\lambda|U_t-\Id|^2)^{-2}\big(U_1XU_1^*U_t-U_t^*U_1XU_1^*\big)\,\d\lambda\,\d t.
\end{align*}
For the last term, we observe that
\begin{align*}
\frac{\d}{\d t}(1+\lambda|U_t-\Id|^2)^{-1}=\lambda(1+\lambda|U_t-\Id|^2)^{-2}(U_tY-YU_t^*)
\end{align*}
so
\begin{align*}
&\frac{\sin(x\pi)}{\pi}\int_0^1\int_0^\infty\lambda^{-x+1}\Tr\big(Y|U_t-\Id|^{2n+2}(1+\lambda|U_t-\Id|^2)^{-2}(U_1XU_1^*U_t-U_t^*U_1XU_1^*)\big)\,\d\lambda\,\d t\\
&=\frac{\sin(x\pi)}{\pi}\int_0^1\int_0^\infty\lambda^{-x}\Tr\big(|U_t-\Id|^{2n+2}\frac{\d}{\d t}\big((1+\lambda|U_t-\Id|^2)^{-1}\big)U_1XU_1^*\big)\,\d\lambda\,\d t.
\end{align*}
Integrating by parts yields
\begin{align*}
&=\frac{\sin(x\pi)}{\pi}\int_0^1\int_0^\infty\lambda^{-x}\Tr\Big(\frac{\d}{\d t}\Big(|U_t-\Id|^{2n+2}(1+\lambda|U_t-\Id|^2)^{-1}\Big)U_1XU_1^*\Big)\,\d\lambda\,\d t\\
&-\frac{\sin(x\pi)}{\pi}\int_0^1\int_0^\infty\lambda^{-x}\Tr\Big(\frac{\d}{\d t}\Big(|U_t-\Id|^{2n+2}\Big)(1+\lambda|U_t-\Id|^2)^{-1}U_1XU_1^*\Big)\,\d\lambda\,\d t\\
&=\frac{\sin(x\pi)}{\pi}\int_0^\infty\lambda^{-x}\Tr\big(|U_1-\Id|^{2n+2}(1+\lambda|U_1-\Id|^2)^{-1}U_1XU_1^*\big)\,\d\lambda\\
&+i\frac{\sin(x\pi)}{\pi}\int_0^1\int_0^\infty\lambda^{-x}\Tr\Big(\big(\sum_{k=0}^n|U_t-\Id|^{2k}(U_tY-U_t^*Y)|U_t-\Id|^{2(n-k)}\big)\times\\
&\qquad\qquad\qquad\qquad\qquad\qquad\qquad\qquad\qquad\qquad\qquad\times(1+\lambda|U_t-\Id|^2)^{-1}U_1XU_1^*\Big)\,\d\lambda\,\d t\\
&=\Tr(X|U_1-\Id|^{2r})
+(n+1)\int_0^1\Tr\big((U_tY-U_t^*Y)|U_t-\Id|^{2n+2x-2}U_1XU_1^*\big)\,\d t\\
&=\Tr\big(X|U_1-\Id|^{2r})+(n+1)\int_0^1\Tr(Y|U_t-\Id|^{2n+2x-2}(U_1XU_1^*U_t-U_t^*U_1XU_1^*)\big)\,\d t.
\end{align*}
Combining our calculations now shows that
\[
\d \Xi_{U_1}(X)=\Tr(X|U_1-\Id|^{2r})
\]
and so normalising by $-i(1/2)^{2r+1}C_r$ we see that $\beta_r$ is exact.
\end{proof}

%
%
%
%
\begin{rmk}
In general the functions $\Theta,\Xi$ depend on the choice of path $e^{tY}$ joining a unitary $U=e^Y$ to the identity. In turn the choice of path depends on the choice of $Y$ such that $U=e^Y$. Near the identity, there is a unique such $Y$, but in general there will exist multiple choices. Nevertheless, the differentials $\alpha,\beta$ of $\Theta,\Xi$ are independent of this choice. We return to this issue in subsection \ref{subsec:paths}.
\end{rmk}

\subsection{Integral formula for the spectral flow of loops of unitaries}
\label{subsec:specflow-unitary}

Having established that the one-forms $\alpha_{U,n}$ and $\beta_{U,r}$ are exact for sufficiently large $n\in \mathbb{N}$ and $r\in\mathbb{R}$, we will now show that the spectral flow for a path of unitaries, as defined in Definition~\ref{defn:BLP-SF}, can be written as an integral of these exact one-forms.

\begin{thm}
\label{thm:cayley-spec-flow-formula}Let $p\geq 1$ and let $(0,1)\ni t\mapsto U_t$ be a Schatten differentiable path in $\mathcal{U}_p$ such that $U_0=U_1=\Id$. Then for $r\geq (p-1)/2$ or $n\geq p-1$ we have 
\begin{align*}
\textup{sf}(U_\bullet)&=\frac{-i}{2^{2r+1}}\frac{\Gamma(r+1)}{\sqrt{\pi}\Gamma(r+1/2)}\int_0^1\Tr\Big(U_t^*\frac{\d}{\d t}(U_t)|U_t-\Id|^{2r}\Big)\,\d t\\
&=(-1)^n\frac{1}{2\pi i}\int_0^1\Tr\Big(U_t^*\frac{\d}{\d t}(U_t)(U_t-\Id)^{n}\Big)\,\d t,
\end{align*}
whenever the integrals converge. 
\end{thm}
\begin{proof}

We will start with the first equality. To this end fix $r\geq (p-1)/2$ and on the Banach manifold $C^1(\mathbb{S}^1,\mathcal{U}_p)$ of  (piecewise) $C^1$-loops in $\mathcal{U}_p$ introduce the function $B_r$ by setting 
 \[
B_r(U_\bullet)=-i\frac{1}{2^{2r+1}}C_r\int_0^1\Tr\Big(U_t^*\frac{\d}{\d t}(U_t)|U_t-\Id|^{2r}\Big)\,\d t
=\int_0^1\beta_r\left(U_t^*\frac{\d}{\d t}(U_t)\right)\,\d t,
\]
where  $[0,1]\ni t\mapsto U_{t,\bullet}$ is a differentiable path in $C^1(\mathbb{S}^1,\mathcal{U}_p)$. 
Note that the chain rule gives
\[
\frac{\d}{\d s}\Big|_{s=s_0}B_r( U_{s,\bullet})=dB_r\Big(\frac{\d}{\d s}\Big|_{s=s_0}U_{s,\bullet}\Big).
\]
Computing the differential at $X\in \L^p$ yields
\[
\d B_{r,U_{s_0,\bullet}}(X)
=\int_0^1\d\beta_r(U_{s_0,t}^*\frac{\d}{\d t}(U_{s_0,t}),X)\,\d t=0
\]
since $\beta_r$ is closed by Lemma \ref{lem:exactness}. Hence $B_r$ is locally constant and so homotopy invariant. Furthermore, the function  $B_r$ satisfies concatenation by construction.

To see that $B_r$ is integer-valued, let $P$ be a rank $k, k\in\mathbb{N},$ projection on the Hilbert space $\H$ and  define the loop
\begin{equation}\label{eq:normilising_loop}
U_t^{(k)}={\rm Id}-P+Pe^{i2\pi t}, t\in [0,1].
\end{equation}
As $k$ eigenvalues go anti-clockwise through $-1$ on this path, the spectral flow for this loop is  $k$. To compute $B_r$ for this loop we note that 
 $\frac{\d}{\d t}U_t^{(k)}=i2\pi Pe^{i\pi t}$ and $U_t^{(k)}-\Id=P(-1+e^{i2\pi t})$.  Hence, we can compute 
\begin{align*}
\int_0^1&\Tr\Big((U_t^{(k)})^*\frac{\d}{\d t}(U_t^{(k)})|U_t^{(k)}-\Id|^{2r}\Big)\,\d t
=\int_0^1\Tr\big(2\pi i P|-1+e^{i2\pi t}|^{2r}\big)\,\d t\\
&=2\pi i k \int_0^1 |-1+e^{i2\pi t}|^{2r}\,\d t=ik\pi 2^{2r+1}\int_0^1\sin^{2r}(t\pi)\,\d t\\
&=ki2^{2r+1}\frac{\sqrt{\pi}\Gamma(r+1/2)}{\Gamma(r+1)},
\end{align*}
where the last line follows from a standard integral. Recalling the definition of the constant $C_r$ (see Definition \ref{defn:one-forms}), we infer that \begin{equation}\label{eq:norm_coincides}
B_r(U_t^{(k)})=k=\textup{sf}(U_t^{(k)}).
\end{equation}

Now $\mathcal{U}_p$ is homotopy equivalent to $\mathcal{U}_1$, and each class in $\pi_1(\mathcal{U}_1)\cong\Z$ is represented by one of the paths $U_t^{(k)}$ or its reverse path. So for any loop $V_\bullet\in\mathcal{U}_p$ we have a homotopy $V_{s,\bullet}$ with $V_{0,\bullet}=V_\bullet$ and $V_{1,\bullet}=U^{(k)}_\bullet$. By the homotopy invariance of $B_r$ and the computation above
\[
B_r(V_{0,\bullet})=B_r(V_{1,\bullet})=k\in\Z.
\]
Referring to the normalisation check in \eqref{eq:norm_coincides}, we see that $B_r=\textup{sf}$ on every loop in $\mathcal{U}_p$. This is analogous to, but simpler than, 
Lesch' uniqueness characterisation of the spectral flow.

To prove the second asserted equality, fix $n\geq p-1$ and introduce the function $A_n$ on the Banach manifold $C^1(\mathbb{S}^1,\mathcal{U}_p)$ by setting 
\[
A_n(U_\bullet)=(-1)^n\frac{1}{2\pi i}\int_0^1\Tr\Big(U_t^*\frac{\d}{\d t}(U_t)(U_t-\Id)^{n}\Big)\,dt=\int_0^1\alpha_n(U_t^*\frac{\d}{\d t}(U_t))\,\d t,
\]
where  $[0,1]\ni t\mapsto U_{t,\bullet}$ be a differentiable path in $C^1(\mathbb{S}^1,\mathcal{U}_p)$. Using an identical proof as for the function $B_r$, one can show that the function is locally constant, homotopy invariant, descends to a map on $\pi_1(\mathcal{U}_p)$ and satisfies concatenation. 
For the simple path $U_t^{(k)}$ introduced in \eqref{eq:normilising_loop} we can check that
\begin{align*}
\int_0^1\Tr\Big((U_t^{(k)})^*\frac{\d}{\d t}(U_t^{(k)})(U_t^{(k)}-\Id)^{n}\Big)\,\d t
&=\int_0^1\Tr\big(2\pi iP(-1+e^{i2\pi t})^{k}\big)\,\d t
=(-1)^n2\pi i k 
\end{align*}
and so we again have $A_n(U_\bullet^{(k)})=k=\textup{sf}(U_\bullet^{(k)})$, showing that the function $A_n$ satisfies the same normalisation as the spectral flow. Thus, $A_n(U_\bullet)=\textup{sf}(U_\bullet)$, proving the second asserted equality. 
%
%
%
%
\end{proof}

In Theorem \ref{thm:cayley-spec-flow-formula}, the integral formula for spectral flow  is presented for a loop of unitaries indexed by the interval $[0,1]$. In many application (particularly those considered in Section \ref{sec:Lev-det} below) it is convenient to consider this formula for paths of unitaries indexed by $[0,\infty)$ or by $\mathbb{R}$. For this reason, we state the version of this formula in this setting. 

\begin{corl}\label{cor:cayley-spec-flow-formula}
Let $p\geq 1$ and let $n\in\mathbb{N}$ and $r\in \mathbb{R}$ be such that $r\geq (p-1)/2$ and $n\geq p-1$. Let $(U_s), s\in [0,\infty)$  be a path in $\mathcal{U}_p$ such that 
\begin{enumerate}
\item $(U_s)$ is $\L^p$-differentiable on any compact interval of $[0,\infty)$;
\item The limit $\lim_{s\to\infty}U_s=:U_\infty$ exists in the Schatten norm $\|\cdot\|_p$ and $U_0=U_\infty=\Id$;
\item The functions $s\mapsto \Tr\Big(U_s^*\frac{\d}{\d s}(U_s)|U_s-\Id|^{2r}\Big)$ and $s\mapsto \Tr\Big(U_s^*\frac{\d}{\d s}(U_s)(U_s-\Id)^{n}\Big)$ are integrable on $[0,\infty)$.
\end{enumerate}
 Then
\begin{align*}
\textup{sf}(U_\bullet)&=\frac{-i}{2^{2r+1}}\frac{\Gamma(r+1)}{\sqrt{\pi}\Gamma(r+1/2)}\int_0^\infty\Tr\Big(U_s^*\frac{\d}{\d s}(U_s)|U_s-\Id|^{2r}\Big)\,\d s\\
&=(-1)^n\frac{1}{2\pi i}\int_0^\infty\Tr\Big(U_s^*\frac{\d}{\d s}(U_s)(U_s-\Id)^{n}\Big)\,\d s.
\end{align*}
A similar result holds for a path indexed by $\mathbb{R}$ (in which case we assume that the limits in the Schatten norm as $s\to\pm \infty$ coincide and are the identity).
\end{corl}
\begin{proof}
 Consider the change of variable  $t=f(s)=1-(1+s)^{-\alpha/2}$, $s\in [0,\infty)$ and the path $V_{f(s)}=U_{s}$. Then  the loop $V_{f(s)}$ is a continuous  loop in $\mathcal{U}_p(H)$ with $V_0=V_1=\Id$. Since $(U_\bullet)$ is differentiable on any compact interval of $[0,\infty)$ we have that $V_\bullet$ is differentiable on $[0,1)$.
 Note that $U'_t=V'_{f(s)}f'(s)$, so that $U^*_tU'_t \frac{dt}{ds}=V^*_sV'_s$. Hence, by Theorem \ref{thm:cayley-spec-flow-formula} we have that 
\begin{align*}
\textup{sf}(U_s)_{s\in[0,\infty)}&=\textup{sf}(V_t)_{t\in[0,1]}\\
&=\frac{-i}{2^{2r+1}}\frac{\Gamma(r+1)}{\sqrt{\pi}\Gamma(r+1/2)}\int_0^1\Tr\Big(V_t^*\frac{\d}{\d t}(V_t)|U_t-\Id|^{2r}\Big)\,\d t\\
&=\frac{-i}{2^{2r+1}}\frac{\Gamma(r+1)}{\sqrt{\pi}\Gamma(r+1/2)}\int_0^\infty\Tr\Big(U_s^*\frac{\d}{\d s}(U_s)|U_s-\Id|^{2r}\Big)\,\d s,
\end{align*}
as required.

For the case when the path is indexed by $\mathbb{R}$ the argument is similar with a possible change of variables given by $t=\frac1{1+e^{-s}}$, $s\in \mathbb{R}$.
\end{proof}

\subsection{Regularised determinants}
\label{subsec:reg-dets}
We now aim to rewrite the right-hand side of the integral formula in Theorem \ref{thm:cayley-spec-flow-formula} in terms of regularised determinants of $U_t$.  We recall some basic definitions and results for regularised determinants and  refer to \cite{GGK, Simon05} for the general theory of regularised determinants.

Let, as before, $\H$ be a separable infinite dimensional Hilbert space.
Let $\{h_n\}_{n=1}^{\infty}$ denote an orthonormal basis for $\H$ and suppose $A\in\L^1$.  For each 
$N \in \N$, let $M_N \in \C^{N \times N}$ denote the matrix with entries
\begin{equation}
\delta_{j,k} + (h_j,Ah_k)_{\H},\quad 1\leq j,k\leq N.
\end{equation}
The sequence $\{ \textup{Det}(M_N)\}_{N=1}^{\infty}$ has a limit as $N \to \infty$, and its value does not depend on the orthonormal basis chosen.  One defines the Fredholm determinant
\begin{equation}
 \textup{Det}(\Id+A):=\lim_{N \to\infty} \textup{Det}(M_N).
\end{equation}
The Fredholm determinant is Fr\'echet differentiable with respect to $A$ (cf., e.g., \cite[Theorem 5.2]{Simon05}).
Moreover, if $\Omega \subseteq \C$ is an open set and $A:\Omega \to \L^1$ is analytic, then the function $ \textup{Det}(\Id + A(\,\cdot\,))$ is analytic in $\Omega$, and  on a dense open subset of $\Omega$
\begin{equation}\label{eq_ddz_logdet}
\frac{d}{dz} \Log\big( \textup{Det}(\Id+A(z))\big) = \tr\big((\Id+A(z))^{-1}A'(z)\big).
\end{equation}
Here we use the convention that $\Log:\C\setminus\{0\}\to\C$ satisfies 
\begin{equation}
\Log(re^{it})=\ln(r)+it\quad\mbox{for}\quad-\pi<t\leq \pi.
\label{eq:log-is-good}
\end{equation} 
This convention means that $\Log\textup{Det}$ has jump discontinuities, but is otherwise smooth.

For an operator $U$ with $U - \textup{Id} \in \L^p$, the operator $U \exp \left( \sum_{\ell = 0}^{p-1} \frac{(-1)^\ell}{\ell} (U-\textup{Id})^\ell \right)- \textup{Id}$ defines a trace-class operator and so we can define the $p$-determinant of $U$ by
\begin{align*}
\textup{Det}_p(U) &= \textup{Det} \left(U \exp \left( \sum_{\ell = 0}^{p-1} \frac{(-1)^\ell}{\ell} (U-\textup{Id})^\ell \right) \right).
\end{align*}
If $p \geq 2$ and $U - \textup{Id} \in \L^{p-1}$ also, we have the recursion formula (see e.g. \cite[Theorem 9.2]{Simon05})
\begin{align}\label{eq:Det-recursion}
\textup{Det}_p(U) &= \textup{Det}_{p-1}(U) \exp\left( \frac{(-1)^{p-1}}{p-1} \textup{Tr} \left( (U-\textup{Id})^{p-1} \right) \right),
\end{align}
and if further $U-{\rm Id}\in \L^q$ for all $1\leq q\leq p$ we have
\begin{align}\label{eq:Det-reduced}
\textup{Det}_p(U) &= \textup{Det}(U) \exp \left( \sum_{\ell = 1}^{p-1} \frac{(-1)^\ell}{\ell} \textup{Tr} \left( (U-\textup{Id})^\ell \right) \right).
\end{align}
%

In the following we shall need a version of Equation \eqref{eq_ddz_logdet} for the regularised determinant $\textup{Det}_p$ and a path of unitaries indexed by a real variable $t$. We start with an auxiliary lemma. The version of this lemma for analytic Schatten-valued maps is well-known (see e.g.   \cite[Lemma 9.1]{Simon05}).
\begin{lemma}\label{lem_derivative_computation}
Let $p\in\mathbb N$, and let $t\mapsto U_t$, $t\in[0,1]$ be a path of $p$-Schatten unitaries which is differentiable in Schatten norm on $(0,1)$. Define $$R_t:=U_t\exp\!\left(\sum_{j=1}^{p-1} \frac{(-1)^j}{j}(U_t-\Id)^j\right)-{\rm Id}.$$
Then $R_t\in \mathcal L^1$, the map $t\mapsto R_t$ is differentiable in the trace-class norm on $(0,1)$, and
$$\Tr\bigl((\Id+R_t)^{-1}R_t'\bigr)
=
(-1)^{p-1}\Tr\bigl(U_t^{-1}(U_t-\Id)^{p-1}U_t'\bigr),
\qquad t\in(0,1).
$$
\end{lemma}

\begin{proof}
Define a scalar function
$$
h(z):=(1+z)\exp\!\left(\sum_{j=1}^{p-1} \frac{(-1)^j}{j}z^j\right)-1,$$
so that $R_t=h(U_t-\Id).$

It is well know that $h(z)=z^pg(z)$ for an entire function $g$ (see e.g. the proof of \cite[Lemma 9.1]{Simon05}). In particular, the H\"older inequality implies that $R_t\in \L^1$. Computing $h'(z)$ yields
\[
h'(z)
=
(1+h(z))
\left(
\frac{1}{1+z}+\sum_{j=1}^{p-1} (-1)^j z^{j-1}
\right).
\]
Note that 
\[
\frac{1}{1+z}+\sum_{j=1}^{p-1} (-1)^j z^{j-1}
=
(-1)^{p-1}\frac{z^{p-1}}{1+z},
\]
and so 
\[
h'(z)=(1+h(z))\,(-1)^{p-1}\frac{z^{p-1}}{1+z},
\]
or, equivalently,
\begin{equation}\label{eq_h_prime}
(1+h(z))^{-1}h'(z)=(-1)^{p-1}\frac{z^{p-1}}{1+z}.
\end{equation}

Next, we show that  $t\mapsto R_t$ is differentiable in the trace norm and compute $R_t'$. We write 
\[
h(z)=\sum_{m=p}^{\infty} a_m z^m,
\]
where $a_m$ are such that $g(z)=\sum_{m=0}^\infty a_mz^m.$ 
For each integer $m\ge 1$, we have 
\[
\frac{d}{dt}(U_t-\Id)^m=\sum_{k=0}^{m-1}(U_t-\Id)^kU_t'(U_t-\Id)^{m-1-k}
\]
in the $\mathcal L^{p/m}$-norm. 
Hence we obtain 
\[
R_t'
=
\sum_{m=p}^{\infty} a_m \sum_{k=0}^{m-1}(U_t-\Id)^kU_t'(U_t-\Id)^{m-1-k}
\]
in $\mathcal L^1$.

Since $\Id+R_t=U_t\exp\!\left(\sum_{j=1}^{p-1} \frac{(-1)^j}{j}(U_t-\Id)^j\right)$ and $U_t$ is unitary, it follows that $\Id+R_t$ is invertible. Therefore
\[
(\Id+R_t)^{-1}R_t'\in \mathcal L^1.
\] 
Using the above formula for $R_t'$, we infer that 
\begin{align*}
\Tr\bigl((\Id+R_t)^{-1}R_t'\bigr)
&=
\sum_{m=p}^{\infty} a_m \sum_{k=0}^{m-1}
\Tr\left((\Id+R_t)^{-1}(U_t-\Id)^kU_t'(U_t-\Id)^{m-1-k}\right)\\
&=\sum_{m=p}^{\infty} m a_m
\Tr\left((\Id+R_t)^{-1}(U_t-\Id)^{m-1}U_t'\right),
\end{align*}
where in the last equation we used the cyclicity of the trace and that $R_t$ commutes with $U_t$. 
Since
\[
h'(z)=\sum_{m=p}^{\infty} m a_m z^{m-1},
\]
we deduce that
\[
\Tr\bigl((\Id+R_t)^{-1}R_t'\bigr)
=
\Tr\left((\Id+R_t)^{-1}h'(U_t-\Id)U_t'\right).
\]
By \eqref{eq_h_prime} we have that

\[
(\Id+R_t)^{-1}h'(U_t-\Id)=(-1)^{p-1}U_t^{-1}(U_t-\Id)^{p-1}.
\]
Substituting this into the trace identity above yields
\[
\Tr\bigl((I+R_t)^{-1}R_t'\bigr)
=
(-1)^{p-1}\Tr\bigl(U_t^{-1}(U_t-\Id)^{p-1}U_t'\bigr),
\]
as required. 
\end{proof}

The following lemma relates spectral flow formula to regularised determinants, which in Section~\ref{sec:Lev-det} will allow us to rewrite Levinson's theorem as a spectral flow formula.

%
\begin{lemma}\label{lem:det-p-to-det}
Suppose that $p \in \N$ and $U_\bullet: [0,1] \to \mathcal{U}_p$ is a path of $p$-Schatten unitaries which is differentiable in $\|\cdot\|_p$-norm on $(0,1)$. 
Then for almost all $t\in[0,1]$ 
\begin{align}\label{eq:Det-log-diff}
\frac{\frac{\d}{\d t}  \textup{Det}_p(U_t) }{\textup{Det}_p(U_t)}
&= (-1)^{p-1} \textup{Tr} \left( U^*_t U'_t (\textup{Id}-U_t)^{p-1} \right).
\end{align}
\end{lemma}
\begin{proof}
Assume first that $p=1$. 
%
%
%
%
Since $U_t$ is trace-class differentiable, we can write 
\[
U_{t+\varepsilon}=U_t+\varepsilon U_t'+O(\varepsilon^2)
\]
as \(\varepsilon\to 0\).
Applying the determinant to both sides and using multiplicativity of the determinant, we obtain
\[
\textup{Det}(U_{t+\varepsilon})
=
\textup{Det}(U_t)\textup{Det}\bigl(\Id+\varepsilon U_t^*U_t'+O(\varepsilon^2)\bigr).
\]
By \cite[p. 268]{Simon77} we have that 
\[
\textup{Det}\bigl(\Id+\varepsilon U_t^*U_t'+O(\varepsilon^2)\bigr)
=
1+\varepsilon \Tr(U_t^*U_t')+O(\varepsilon^2).
\]
Substituting this into the previous expression yields
\[
\textup{Det}(U_{t+\varepsilon})
=
\textup{Det}(U_t)+\varepsilon \textup{Det}(U_t)\Tr(U_t^*U_t')+O(\varepsilon^2),
\]
so that 
\[
\frac{\textup{Det}(U_{t+\varepsilon})-\textup{Det}(U_t)}{\varepsilon\,\textup{Det}(U_t)}
=
\Tr(U_t^*U_t')+O(\varepsilon).
\]
Passing to the limit as \(\varepsilon\to 0\), it follows that
\[
\frac{\frac{d}{dt}\textup{Det}(U_t)}{\textup{Det}(U_t)}=\Tr(U_t^{-1}U_t'),
\]
which proves the formula for $p=1$.

Let $p\in \N$ be arbitrary. Arguing similarly we have that  there is a decomposition of the interval $[0,1]$ into finitely many closed intervals intersecting only in singletons such that $\Log\textup{Det}_p(U_t)$ is differentiable on each interval. Let $R_t$ as in Lemma \ref{lem_derivative_computation}. By Lemma \ref{lem_derivative_computation}, $R_t$ is a trace-class differentiable function with 
\[
\Tr(R_t^{-1}R_t')=(-1)^{p-1}\Tr\bigl(U_t^{-1}(U_t-\Id)^{p-1}U_t'\bigr).
\] 
Therefore, since $\textup{Det}_p(U_t)=\textup{Det}(R_t)$ we infer that 
\[
\frac{\frac{d}{dt}\textup{Det}_p(U_t)}{\textup{Det}_p(U_t)}=\Tr(R_t^{-1}R_t')=(-1)^{p-1}\Tr\bigl(U_t^{-1}(U_t-\Id)^{p-1}U_t'\bigr),
\]
as required. 
\end{proof}

Combining Theorem \ref{thm:cayley-spec-flow-formula} and Lemma \ref{lem:det-p-to-det} we have the following result. 

\begin{thm}
\label{thm:spec-flow-formula-det}
Suppose that $p \in \N$ and $U_\bullet: [0,1] \to \mathcal{U}_p$ is a $C^1$ loop of $p$-Schatten  unitaries. Then
\begin{align*}
\textup{sf}(U_\bullet)&=\frac{1}{2\pi i}\int_0^1\frac{\frac{d}{dt}\textup{Det}_p(U_t)}{\textup{Det}_p(U_t)}\,\d t.\end{align*}
\end{thm}
\begin{proof}
Comparing the integral with the result of Theorem \ref{thm:cayley-spec-flow-formula} completes the proof.
\end{proof}

We record the following result which follows immediately from equality \eqref{eq:Det-reduced}.

\begin{lemma}\label{lem:logdet_p_lodget}
Suppose that $U_\bullet: [0,1] \to \mathcal{U}_1$ is a trace-class path of unitaries. Then 
$$\frac{\frac{d}{dt}\textup{Det}_p(U_t)}{\textup{Det}_p(U_t)}
=
\frac{\frac{d}{dt}\textup{Det}(U_t)}{\textup{Det}(U_t)}
+ \frac{d}{dt}\sum_{\ell=1}^{p-1} \frac{(-1)^\ell}{\ell}
\Tr\big((U_t-\Id)^\ell\big).$$
\end{lemma}

\subsection{Spectral flow for general paths}
\label{subsec:paths}

So far we have considered spectral flow for loops of unitaries. In this section, we consider spectral flow for path of $p$-Schatten unitaries. Here the idea is to connect the endpoints of the path to the identity via a simple path. 

Let $U_\bullet$ be a $C^1$-path of unitaries in $\mathcal{U}_p$ for some $p\geq 1$ which is not closed (i.e. not a loop). Set $U_0=e^Y$, $U_1=e^Z$ for some skew $Y,Z\in\L^p$. Define paths $V_t=e^{tY}$ and $W_t=e^{(1-t)Z}$ for $t\in[0,1]$. The choice of $Y,Z$ is not unique and we address the choices below.

\begin{defn}
\label{defn:spec-flow-path}
With $U_\bullet$ a path in $\mathcal{U}_p$ and $V_\bullet$, $W_\bullet$ as above, we define
\[
\textup{sf}_{Y,Z}(U_\bullet):=\textup{sf}(V_\bullet*U_\bullet*W_\bullet)
\]
where $*$ denotes concatenation.
\end{defn}

\begin{prop}
\label{prop:path-sf-formula}
Recall $\Theta,\Xi$ defined in Equations \eqref{eq:they-called-him-Theta} and \eqref{eq:they-called-him-Xi} respectively.
With $U_\bullet$, $V_\bullet$, $W_\bullet$ as in Definition \ref{defn:spec-flow-path}, for any $n\geq p-1$ we have the formulae
\begin{align*}
\textup{sf}_{Y,Z}(U_\bullet)
&=\int_0^1\alpha(U^*_t\dot{U}_t)\,dt+\Theta(U_0)-\Theta(U_1)\\
&=\int_0^1\beta(U^*_t\dot{U}_t)\,dt+\Xi(U_0)-\Xi(U_1)\\
&=\frac{(-1)^n}{2\pi i} \int_0^1 \Tr(U_t^*\dot U_t(U_t-\Id)^{n}) dt+\frac{(-1)^n}{2\pi i} \int_0^1 \Tr(V_t^*\dot V_t(V_t-\Id)^{n}) dt\\
&\qquad+\frac{(-1)^n}{2\pi i} \int_0^1 \Tr(W_t^*\dot W_t(W_t-\Id)^{n}) dt\\
&=\frac1{2\pi i}\Big[\int_0^1\frac{\frac{d}{dt}\textup{Det}_p(U_t)}{\textup{Det}_p(U_t)}\,dt+\int_0^1\frac{\frac{d}{dt}\textup{Det}_p(V_t)}{\textup{Det}_p(V_t)}\,dt+\int_0^1\frac{\frac{d}{dt}\textup{Det}_p(W_t)}{\textup{Det}_p(W_t)}\,dt\Big].
\end{align*}
\end{prop}
\begin{proof}
These follow from 
Theorems \ref{thm:cayley-spec-flow-formula} and \ref{thm:spec-flow-formula-det} 
and the relations $d\Theta=\alpha$ and $d\Xi=\beta$.
\end{proof}

The functions $\Theta,\Xi$ play the same role as the eta and kernel correction terms \cite{CP1,CPotSuk} in spectral flow formulae for paths of self-adjoint Fredholm operators.

\begin{rmk}
\label{rmk:path-convention}
We now address the dependence of spectral flow for a (non-closed) path on the choice of the paths connecting endpoints to the identity. In the self-adjoint case, any two choices of admissible paths in the affine space between two fixed points are homotopic \cite{CP1}. In particular, there is no ambiguity when connecting two endpoints of the path to the same point. In the case of unitaries, this is no longer the case and two different loops starting and ending at the identity may have different spectral flow. 

In Definition \ref{defn:spec-flow-path} we have made a choice to connect the endpoints to the identity via geodesics. 
 Recall, that the exponential map $\exp:T\mathcal{U}_p\to \mathcal{U}_p$ is surjective, but not injective. Thus in general our definition of spectral flow for a non-closed path depends on how we represent the endpoints. 
 
 Nevertheless, when the endpoints are finite rank perturbations of the identity, we can make a convention to fix these choices. The tangent cut locus of $U(n)$ is \cite[Section 3]{Sakai77}
\[
Ad_G\Big\{\frac{i\pi}{X}\sum_{j=1}^nx_jE_{jj}:\,\sum_jx_j^2=1,\ X=\max_j|x_j|\Big\}\subset T_{{\rm Id}}U(n)
\]
where the $E_{jj}$ are the standard diagonal matrix units. 

Given a unitary finite rank perturbation of the identity $U$, we may diagonalise $U$ to be of the form ${\rm diag}(e^{i\theta_1},\dots,e^{i\theta_n},1,1,\dots)$. If $U$ is a cut point, then at least one $e^{i\theta_j}=-1$. We then choose to represent this point as $-1=e^{i\pi}$ (not $e^{-i\pi}$).
This choice is consistent with the (discontinuous) principal branch of $\Log$ extended to the negative real axis from above as in Equation \eqref{eq:log-is-good}. We will adopt this normalisation in the next section.
\end{rmk}

%
%
%

\section{Application to scattering theory}
\label{sec:Lev-det}

In this section we apply spectral flow for unitaries to the scattering operator $S$ of Schr\"{o}dinger systems. In this setting, the scattering matrix $S(\lambda)$, $\lambda\in[0,\infty)$  is a differentiable path of $p$-Schatten unitaries. Using Theorem \ref{thm:cayley-spec-flow-formula}, we will reformulate  Levinson's theorem as a spectral flow formula.  The interesting feature here is that as a function of energy, $\lambda\mapsto S(\lambda)-{\rm Id}$ is pointwise trace-class, but the trace of $S^*(\lambda)S'(\lambda)$ is not integrable. This makes it impossible to apply the classical winding number formula for the spectral flow in this setting. Including additional powers of $S(\lambda)-\Id$ in the integral formula for the spectral flow provides the appropriate correction.  

\subsection{Scattering theory for Schr\"{o}dinger operators}

Consider the free Laplacian $H_0=-\sum_j\partial_j^2$ acting on $L^2(\R^d)$ 
with domain the second Sobolev space $H^{2}(\R^d)$. For $V \in C_c^\infty(\R^d)$ we can define the Schr\"{o}dinger operator 
$H = H_0 +V$,
also with domain $H^{2}(\R^d)$. The spectrum of $H_0$ is entirely absolutely continuous and given by $\sigma_{ac}(H_0) = [0,\infty)$ and the spectrum of $H$ consists of the absolutely continuous part $\sigma_{ac}(H) = \sigma_{ac}(H_0) = [0,\infty)$ and a finite number of non-positive eigenvalues with finite multiplicity.

For $z \in \C \setminus \R$ we denote the resolvent of $H_0$ by
$R_0(z) = (H_0-z)^{-1}$,
considered  as a bounded operator from $H^{0}(\R^d)$ to $H^{2}(\R^d)$. 
Let $v = |V|^\frac12$ and $u = \textup{sign}(V)$. Then for $\lambda > 0$ the limiting absorption principle \cite[Theorems 4.1 and 4.2]{agmon75} tells us that the limits
\begin{align}\label{eq:SLAP}
v R_0(\lambda \pm i0) v  &:= \lim_{\eps \to 0} vR_0(\lambda \pm i \eps) v \quad \textup{ and } \quad (u+vR_0(\lambda+i0) v)^{-1} := \lim_{\eps \to 0} (u+vR_0(\lambda+i\eps) v)^{-1}
\end{align}
exist in the topology of $\B(L^2(\R^d))$. 

The wave operators are
$W_\pm := \mathop{\slim}_{t \to \pm \infty} \e^{itH} \e^{-it H_0}$,
and satisfy the following properties (see \cite[Theorem 1.6.2]{yafaev10} and \cite[Proposition 2.1]{alexander24}).
\begin{lemma}
The wave operators are isometries and  $\textup{Range}(W_\pm) = P_{ac}(H)L^2(\R^d)$. Hence they are Fredholm operators with index
\begin{align}
\textup{Index}(W_\pm) &= -N,
\end{align}
where $N$ is the total number of eigenvalues (counted with multiplicity) of $H$.
Moreover the scattering operator 
$S = W_+^* W_-$
is a unitary operator which commutes with $H_0$
\end{lemma}

Define the unitary operator $F_0: L^2(\R^d) \to L^2(\R^+) \otimes L^2(\Sf^{d-1})$ by
\begin{align}
[F_0 f](\lambda,\omega) &= 2^{-\frac12} \lambda^{\frac{d-2}{4}} (2\pi)^{-\frac{d}{2}} \int_{\R^d} \e^{-i \lambda^\frac12 \langle x, \omega \rangle} f(x) \, \d x.
\end{align}
The operator $F_0$ diagonalises $H_0$ as multiplication by the spectral variable in $\R^+$ (see \cite[Section 1.2]{yafaev10}), and so there exists a family of unitary operators $\{S(\lambda)\}_{\lambda \in \R^+}$ acting on $L^2(\Sf^{d-1})$, defined by the relation
$[S(\lambda) f](\lambda,\omega) = [F_0 S F_0^* f](\lambda,\omega)$.
The operator $S(\lambda)$ is called the scattering matrix at energy $\lambda$
 and has the representation
\begin{align}\label{eq:stat-scat-mat}
S(\lambda) &= \textup{Id} - 2\pi i F_0(\lambda) v (u + vR_0(\lambda+i0) v)^{-1} v F_0(\lambda)^*,
\end{align}
where  $F_0(\lambda)v: L^2(\R^d)\to L^2(\Sf^{d-1})$ is given by $
[F_0(\lambda)vf](\omega) = [F_0 vf](\lambda,\omega)$.

In the following proposition we collect some basic properties of the scatting matrix ensuring that it is a path of unitaries for which we can consider the spectral flow formulae. 

\begin{prop}\cite[Proposition 8.1.5 and Proposition 8.1.9]{yafaev10}
\label{prop:SM-basic}
Suppose that $V \in C_c^\infty(\R^d)$. We have that $S(\lambda)-\textup{Id} \in \L^p(L^2(\Sf^{d-1}))$ for all $p \in \N$ and $S(\lambda)$ is $k$-times differentiable in $\L^p(L^2(\Sf^{d-1}))$ for any $k\in \mathbb{N}$.  Furthermore,
\begin{align}\label{eq:SM_high_energy_estimate}
\| S(\lambda)-\textup{Id}\|_p &\leq C \lambda^{-\frac12+\frac{d-1}{2p}},
\end{align}
for some constants $C>0$ and $\lambda\geq \lambda_0>0$. In particular,  $S(\lambda) \to \textup{Id}$ in $\L^q(L^2(\Sf^{d-1}))$ for all $q \geq d$ and the function $\lambda\mapsto \Tr_{L^2(\Sf^{d-1})}((S(\lambda) - \textup{Id})^p)$ is absolutely integrable on $[0,\infty)$ for $p>d+1$. 
\end{prop}

By Proposition \ref{prop:SM-basic} we have that $S(\lambda)$ is a locally differentiable path of $\L^1(L^2(\Sf^{d-1}))$-class unitaries, but the limit $S(\infty)=\Id$ holds only in $\L^d(L^2(\mathbb{S}^{d-1}))$.

The final item we require before we can apply the spectral flow formula of Theorem \ref{thm:cayley-spec-flow-formula} is the low energy behaviour of $S(\bullet)$. %
The low energy behaviour is dictated by the presence (or absence) of zero energy resonances of $H$, namely distributional solutions of $H\psi=0$, and whether $d\geq 2$. 
%
\begin{defn}\label{defn:resonances}
Suppose that $V \in C_c^\infty(\R^d)$. If $d \neq 2$ we say there is an $s$-resonance if there exists a non-zero bounded distributional solution to $H\psi = 0$. If $d = 2$ we say there is a $p$-resonance if there exists a non-zero distributional solution $\psi$ to $H\psi = 0$ with $\psi \in L^q(\R^2) \cap L^\infty(\R^2)$ for some $q > 2$. We say that there is an $s$-resonance if there exists a non-zero bounded distributional solution $\psi$ to $H\psi = 0$ with $\psi \notin L^q(\R^2)$ for all $q < \infty$.
\end{defn}
General bounds on the resolvent of $H$ \cite{jensen80} show that there can be no resonances for dimension $d \geq 5$. The naming convention for the types of resonances is due to the fact that for spherically symmetric potentials, $s$-resonances can only occur with angular momentum zero and $p$-resonances with angular momentum one (see \cite{bolle88, KS80}).

The following result describes the low energy behaviour of the scattering matrix in the norm of $\B(L^2(\Sf^{d-1}))$. The low energy behaviour of the scattering operator has been obtained in dimension $d = 1$ in \cite[Theorem 4.1]{bolle85} (see also \cite[Proposition 9]{kellendonk08} and \cite[Theorem 2.15]{zworski19}), in dimension $d = 2$ in \cite[Theorem 4.3]{bolle88} and \cite[Theorem 1.1]{richard21}, in dimension $d = 3$ in \cite[Section 5]{jensen79} and in dimension $d \geq 4$ in \cite[Theorem 2.15]{AR23}. 
\begin{thm}\label{thm:scat-mat-zero}
Suppose that $V \in C_c^\infty(\R^d)$. Then for $d \neq 1,3$ we have $S(0) = \textup{Id}$. For $d = 3$ we have $S(0) = \textup{Id}-2P_s$, where $P_s$ is the projection onto spherical harmonics of order zero in $L^2(\Sf^2)$ if there exists a resonance and $P_s = 0$ otherwise. For $d = 1$ we have
\begin{align*}
S(0) &= \begin{pmatrix} 0 & -1 \\ -1 & 0 \end{pmatrix}
\end{align*}
if there does not exist a resonance and if there does exist a resonance, there exists $c_+, c_- \in \R \setminus\{0\}$ with $c_+^2+c_-^2 = 1$ and
\begin{align*}
S(0) &= \begin{pmatrix} 2 c_+ c_- & c_+^2-c_-^2 \\ c_-^2 - c_+^2 & 2 c_+ c_- \end{pmatrix}.
\end{align*}
\end{thm}

In the following, we shall require a little more detailed analysis, showing that the values of $S(0)$ in Theorem \ref{thm:scat-mat-zero} are, in fact, the trace-class limit of $S(\lambda)$ as $\lambda\to 0$. The result below is folklore and we briefly outline the proof for completeness. 

\begin{prop}Let, as before, $V\in C_c^\infty(\mathbb{R}^d)$. Suppose that $d\geq 2$ and $d\neq 3$. Then $\lim_{\lambda\to 0} S(\lambda)=\Id$ in the trace-class norm. If $d=3$, then $\lim_{\lambda\to 0} S(\lambda)=\Id-2P_s$ in the trace-class norm. For $d=1$, $S(0)=\lim_{\lambda\to 0}S(\lambda)$ in the trace-class norm.
\end{prop}
\begin{proof}Since for $d=1$ the scattering matrix is a finite-dimensional marix, we immediately have the equality in the trace-class norm.

By equality \eqref{eq:stat-scat-mat} we have that 
$$
S(\lambda)-\Id =- 2\pi i F_0(\lambda) v (u + vR_0(\lambda+i0) v)^{-1} v F_0(\lambda)^*.
$$

It follows from \cite[equality 1.2.5]{yafaev10} that for  \(f\in L^2(\mathbb{R}^d)\), we have that 
\[
(F_0(\lambda) v f)(\omega)
=\frac{1}{\sqrt{2}}\,\lambda^{\frac{d-2}{4}}
\,\widehat{(vf)}(\sqrt{\lambda}\,\omega).
\]
Therefore, the integral kernel $K(\lambda;\omega,x)$ of the operator $F_0(\lambda) v$ is given by 
\[
K(\lambda;\omega,x)
=\frac{1}{\sqrt{2}}(2\pi)^{-d/2}
\lambda^{\frac{d-2}{4}}
e^{-i\sqrt{\lambda}\,\omega\cdot x}v(x).
\]
A direct computation implies that the Hilbert-Schmidt norm can be estimated as 
\begin{equation}\label{eq_F_0_lambda_at_zero}
\|F_0(\lambda)v\|_2= \mathcal{O}(\lambda^{\frac{d-2}{4}})\to 0,\quad \lambda\to 0.
\end{equation}

If $d\geq 5$, then the operator
$(u + vR_0(\lambda+i0) v)^{-1}$ has asymptotic expansion as $\lambda\to 0$ in the operator norm satisfying $(u + vR_0(\lambda+i0) v)^{-1}=\mathcal{O}(\frac1\lambda)$  (see \cite{AR23}). Therefore, the asymptotic in \eqref{eq_F_0_lambda_at_zero} guarantees that $S(\lambda)\to \Id$ in the trace-class norm. 

For dimensions $2\leq d\leq 4$ the key is the expansion of $ (u + vR_0(\lambda+i0) v)^{-1}$ as $\lambda\to 0$.
By the threshold expansions of $(u + vR_0(\lambda+i0) v)^{-1}$, one may write
\[
(u + vR_0(\lambda+i0) v)^{-1}=L(\lambda)+G(\lambda),
\]
where $L(\lambda)$ denotes the leading part and $G(\lambda)$ is finite rank for $\lambda$ near $0$ (see \cite{bolle88} and \cite{richard21} for $d=2$,  \cite[Section 5]{jensen79} for $d=3$, and \cite{AR23} for $d=4$). 
Hence
\[
S(\lambda)-I
=
-2\pi i\,F_0(\lambda)\,v\,L(\lambda)\,v\,F_0(\lambda)^*
-
2\pi i\,F_0(\lambda)\,v\,G(\lambda)\,v\,F_0(\lambda)^*.
\]

The behaviour of the first term is already known from the proofs of the operator-norm
limits of $S(\lambda)$ at threshold: in dimensions $d=2$ and $d=4$ it tends to $0$ ( \cite[Theorem 4.3]{bolle88} and \cite[Theorem 1.1]{richard21} for $d=2$ and  \cite[Theorem 2.15]{AR23} for $d=4$),
while in dimension $d=3$ it converges to the finite-rank resonance correction
$-2P_s$ \cite[Section 5]{jensen79}. Thus no additional argument is needed for the leading part.

For the second term, since $G(\lambda)$ is finite rank, the operator
\[
F_0(\lambda)\,v\,G(\lambda)\,v\,F_0(\lambda)^*
\]
is finite rank as well. Therefore, the existing operator-norm convergence results in fact prove trace-norm convergence. 
It follows that the operator-norm threshold limits of $S(\lambda)$ automatically upgrade
to trace-norm limits:
\[
S(\lambda)\to I \quad\text{in }\mathcal L_1 \ \text{for } d=2,4,
\]
and
\[
S(\lambda)\to I-2P_s\quad\text{in }\mathcal L_1 \ \text{for } d=3.\qedhere
\]
\end{proof}

\subsection{Determinants of the scattering operator}

To relate Levinson's theorem to spectral flow we firstly need to establish certain formulas for regularised  determinants of the scattering matrix.
%
%
%

By Proposition \ref{prop:SM-basic}, we have that $S(\lambda)-\textup{Id} \in \L^p(L^2(\Sf^{d-1}))$ for all $p \in \N$, and so the determinant of $S(\lambda)$ is well-defined. We have the following more precise statement, due to Guillop\'{e} \cite[Theorem III.1]{guillope81}.

\begin{lemma}\label{lem:guillope-det-S}
Suppose that $V = q_1 q_2$ with $q_1,q_2 \in C_c^\infty(\R^d)$. Then for all $\lambda > 0$ and $p \geq \lfloor \frac{d}{2} \rfloor$ we have
\begin{align*}
&\textup{Det} \left(S(\lambda) \right) \\
&= \frac{\textup{Det}_p \left( \textup{Id}+q_1 R_0(\lambda - i 0 ) q_2 \right)}{\textup{Det}_p \left( \textup{Id}+q_1 R_0(\lambda+i0) q_2 \right)} \exp \left( \sum_{\ell = 1}^{p-1} \frac{(-1)^\ell}{\ell} \textup{Tr} \left( \left(q_1 R_0(\lambda+i0) q_2 \right)^\ell - \left(q_1 R_0(\lambda-i0) q_2 \right)^\ell \right) \right).
\end{align*}
\end{lemma}

We also recall the following \cite[Propositions 7.1.17, 7.1.22, 9.1.2 and 9.1.3]{yafaev10}
\begin{lemma}\label{lem:det-limits}
Suppose that $V = q_1 q_2$ with $q_1, q_2 \in C_c^\infty(\R^d)$. If  $d = 2,3$ let $p \geq 2$ and if $d \geq 4$ let $p \geq d$. Define for $z \in \C \setminus \R$ the function
\begin{align*}
D_p(z) = \textup{Det}_p(\textup{Id}+q_1 R_0(z) q_2).
\end{align*}
Then we have
$\lim_{|z| \to \infty} D_p(z) = 1$
uniformly in $\textup{Arg}(z)$. By the limiting absorption principle \eqref{eq:SLAP} this limit extends to the positive real axis also. Furthermore, we have $D_p(\lambda \pm i0)$ is continuous in $\lambda$.
\end{lemma}

Combining Lemmas \ref{lem:guillope-det-S} and \ref{lem:det-limits} with equation \eqref{eq:Det-reduced} we obtain the following result.
\begin{prop}\label{prop:scat-limit}
Suppose that $V = q_1 q_2$ with $q_1, q_2 \in C_c^\infty(\R^d)$ and let $S(\lambda)$ be the associated scattering matrix at energy $\lambda$. If  $d = 2,3$ let $p \geq 2$ and if $d \geq 4$ let $p \geq d$. Then 
\begin{align}
\lim_{\lambda \to \infty} \textup{Tr} \left(\sum_{\ell = 1}^{p-1} \frac{(-1)^\ell}{\ell}  (S(\lambda)-\textup{Id})^\ell +\sum_{\ell = 1}^{p-1} \frac{(-1)^\ell}{\ell} \textup{Tr} \left( \left(q_1 R_0(\lambda+i0) q_2 \right)^\ell - \left(q_1 R_0(\lambda-i0) q_2 \right)^\ell \right) \right)&= 0.
\end{align}
\end{prop}
\begin{proof}
By Proposition \ref{prop:SM-basic} we have  that $S(\lambda) \to \textup{Id}$ in $\L^q(L^2(\Sf^{d-1}))$ for all $q \geq n$. Hence for $p \geq d$, we have that 
$\lim_{\lambda \to \infty} \textup{Det}_p(S(\lambda)) = 1$.
Since $S(\lambda)$ is a compact perturbation of the identity and can only have eigenvalues accumulating at $1$, we further have that
\begin{align*}
\lim_{\lambda \to \infty} \ln \textup{Det}_p(S(\lambda)) &= 0.
\end{align*}
Combining Lemma \ref{lem:guillope-det-S} with  equation \eqref{eq:Det-reduced} we obtain the relation
\begin{align*}
& \textup{Det}_p(S(\lambda)) \\
&=  \textup{Det}(S(\lambda)) \exp \left( \sum_{\ell = 1}^{p-1} \frac{(-1)^\ell}{\ell} \textup{Tr} \left( (S(\lambda)-\textup{Id})^\ell \right) \right) \\
&=  \frac{\textup{Det}_p \left( \textup{Id}+q_1 R_0(\lambda - i 0 ) q_2 \right)}{\textup{Det}_p \left( \textup{Id}+q_1 R_0(\lambda+i0) q_2 \right)} \exp \left( \sum_{\ell = 1}^{p-1} \frac{(-1)^\ell}{\ell} \textup{Tr} \left( \left(q_1 R_0(\lambda+i0) q_2 \right)^\ell - \left(q_1 R_0(\lambda-i0) q_2 \right)^\ell \right) \right) \times \\
& \exp \left( \sum_{\ell = 1}^{p-1} \frac{(-1)^\ell}{\ell} \textup{Tr} \left( (S(\lambda)-\textup{Id})^\ell \right) \right).
\end{align*}
By Lemma \ref{lem:det-limits} we have that 
\begin{align*}
\lim_{\lambda \to \infty} \frac{\textup{Det}_p \left( \textup{Id}+q_1 R_0(\lambda - i 0 ) q_2 \right)}{\textup{Det}_p \left( \textup{Id}+q_1 R_0(\lambda+i0) q_2 \right)} &= 1.
\end{align*}
Thus we obtain also the limit
\begin{align*}
\lim_{\lambda \to \infty} \exp \left( \sum_{\ell = 1}^{p-1} \Big(\frac{(-1)^\ell}{\ell}  \textup{Tr} \left((S(\lambda)-\textup{Id})^\ell \right)  +  \textup{Tr} \left( \left(q_1 R_0(\lambda+i0) q_2 \right)^\ell - \left(q_1 R_0(\lambda-i0) q_2 \right)^\ell \Big)\right) \right) &= 1.
\end{align*}
Since $q_1 R_0(\lambda \pm i 0)q_2$ are compact operators, the operators $\textup{Id}+q_1 R_0(\lambda \pm i 0) q_2$ can only have eigenvalues accumulating at $1$ \cite[Corollary 8.8.2]{yafaev10}, from which it follows that
\begin{align*}
\lim_{\lambda \to \infty} \ln \textup{Det}_p \left( \textup{Id}+q_1 R_0(\lambda \pm i 0 ) q_2 \right) &= 0.
\end{align*}
Finally, since $S(\lambda)$ has eigenvalues accumulating only at $1$, we have
\begin{align*}
\lim_{\lambda \to \infty} \ln \exp \left( \sum_{\ell = 1}^{p-1} \frac{(-1)^\ell}{\ell}  \textup{Tr} \left((S(\lambda)-\textup{Id})^\ell \right)  \right) &= 0,
\end{align*}
from which the result follows.
\end{proof}
From \cite[Theorem 3.13]{alexander24} we have the following estimate, which when combined with Lemma~\ref{lem:guillope-det-S} and Proposition \ref{prop:scat-limit} shows that $\textup{Det}(S(\lambda))$ does not necessarily have a limit  as $\lambda\to\infty$.
\begin{lemma}
\label{lem:asymptotic}
Suppose that $q_1, q_2 \in C_c^\infty(\R^d)$ and $V = q_1 q_2$. Let $p$ satisfy the assumptions of Proposition \ref{prop:scat-limit}. Then for all $J \in \N$ and $1 \leq j \leq  J$ there exists $C_j(d,V) \in \C$ such that for sufficiently large $\lambda$ we have the estimate
\begin{align*}
\left|\sum_{\ell = 1}^{p-1} \frac{(-1)^\ell}{\ell} \textup{Tr} \left( \left(q_1 R_0(\lambda+i0) q_2 \right)^\ell - \left(q_1 R_0(\lambda-i0) q_2 \right)^\ell \right) + \sum_{j=1}^J C_j(d,V) \lambda^{\frac{d}{2}-j} \right| &\leq C \lambda^{\frac{d}{2}-J-3}.
\end{align*}
\end{lemma}

The constants $C_j(d,V)$ are directly related to the expansion of the heat trace $\textup{Tr} \left( \e^{-tH}-\e^{-tH_0} \right)$ as $t \to 0^+$. For our statement of Levinson's theorem in Section \ref{sec:levinson-statement} we need polynomials in the energy with the $C_j(d,V)$ as coefficients as follows.

\begin{defn}\label{defn:high-energy-poly1}
We define the high-energy polynomial $P_d$ by
\begin{align}
P_d(\lambda) &= \sum_{j=1}^{\lfloor \frac{d}{2} \rfloor} C_j(d,V) \lambda^{\frac{d}{2}-j}.
\end{align}
We use the notation $p_d(\lambda) := P_d'(\lambda)$ also.
\end{defn}

We can explicitly compute the lowest order polynomials (see \cite[Section 4]{alexander24}, finding $P_1 = 0$,
\begin{align*}
P_2(\lambda) &= -\frac{ (2\pi i) \textup{Vol}(\Sf^1)}{2(2\pi)^2} \int_{\R^2} V(x) \, \d x, \\
P_3(\lambda) &= -\frac{(2\pi i) \lambda^\frac12 \textup{Vol}(\Sf^2)}{2(2\pi)^3} \int_{\R^3} V(x) \, \d x,  \\
P_4(\lambda) &= - \frac{(2\pi i) \lambda \textup{Vol}(\Sf^3)}{2(2\pi)^4}\! \int_{\R^4} \!V(x)  \d x +\! \frac{(2\pi i) \textup{Vol}(\Sf^3)}{4(2\pi)^4}\! \int_{\R^4}\! V(x)^2 \d x.
\end{align*}

With these notations at hand, we can reformulate Proposition \ref{prop:scat-limit} and  Lemma \ref{lem:asymptotic} as follows.

\begin{prop}
\label{prop:asymptotic_2}
Suppose that $q_1, q_2 \in C_c^\infty(\R^d)$ and $V = q_1 q_2$. Let $p$ satisfy the assumptions of Proposition \ref{prop:scat-limit}. Then there exists a polynomial $P_d$ such that for sufficiently large $\lambda$ we have the estimate
\begin{align*}
\lim_{\lambda\to \infty}\Big(\textup{Tr} \Big(\sum_{\ell = 1}^{p-1} \frac{(-1)^\ell}{\ell}  (S(\lambda)-\textup{Id})^\ell\Big)+P_d(\lambda)\Big)=0.
\end{align*}
\end{prop}

\subsection{Levinson's theorem and spectral flow}\label{sec:levinson-statement}

The following version of Levinson's theorem is the result of many papers by numerous authors, for instance \cite{alexander24, ANRR, AR23-4D, bolle88, bolle77, levinson49}, with the precise statement below taken from \cite[Theorem 5.10]{alexander24}.

\begin{thm}\label{thm:Levinson}
Suppose that $V \in C_c^\infty(\R^d)$ with corresponding scattering operator $S$, and that there are no zero energy resonant states and that $d\geq2$. Then the number of bound states $N$ of $H = H_0+V$ is given by
\begin{align}
- N & =   \frac{1}{2\pi i} \int_0^\infty \left( \textup{Tr}\left(S(\lambda)^*S'(\lambda) \right) - p_d(\lambda) \right) \, \d \lambda - \frac{1}{2\pi i} P_d(0)+N_{res},
\end{align}
where $p_d$ and $P_d$ are as in Definition \ref{defn:high-energy-poly1}
 and 
\begin{align*}
N_{res} &= \begin{cases} \frac12 & \quad \textup{ if } d=1 \textup{ and there are no resonances,} \\
\textup{number of }p\textup{ resonances} & \quad \textup{ if } d = 2, \\
\frac12 & \quad \textup{ if } d = 3 \textup{ and there exists a resonance}, \\
\textup{number of }s\textup{ resonances} & \quad \textup{ if } d = 4, \\
0 & \quad \textup{ otherwise}.
\end{cases}
\end{align*}
\end{thm}

\bigskip

\begin{lemma}\label{lem_integral_comp}
Suppose that $q_1, q_2 \in C_c^\infty(\R^d)$ and $V = q_1q_2$. Then the integral 
$$
\int_0^\infty \Tr\!\big(S^*(\lambda) S'(\lambda) (S(\lambda)-\Id)^{d-1}\big)\, d\lambda
$$ 
converges and 
\begin{align*}
&(-1)^{d-1}\frac{1}{2\pi i}\int_0^\infty \Tr\big(S^*(\lambda) S'(\lambda) (S(\lambda)-\Id)^{d-1}\big)\, d\lambda
=(-1)^{d-1}\frac{1}{2\pi i}\int_0^\infty\frac{\frac{d}{d\lambda}\Det_d(S(\lambda))}{\Det_d(S(\lambda))}\\
&=\frac{1}{2\pi i} \int_0^\infty \left( \textup{Tr}\left(S(\lambda)^*S'(\lambda) \right) - p_d(\lambda) \right) \, \d \lambda-\frac1{2\pi i}\sum_{l=1}^{d-1}  \frac{(-1)^\ell}{\ell}
\Tr\big((S(0)-\Id)^\ell\big)- \frac{1}{2\pi i} P_d(0).
\end{align*}
\end{lemma}
\begin{proof}
We introduce the notation 
$$
H_d(\lambda) := \sum_{\ell=1}^{d-1} \frac{(-1)^\ell}{\ell}
\Tr\big((S(\lambda)-I)^\ell\big), \lambda\in [0,\infty).
$$

Since for any $\lambda\in[0,\infty)$ we have that $S(\lambda)-\Id$ is a trace-class operator, Lemma \ref{lem:det-p-to-det}  and Lemma \ref{lem:logdet_p_lodget} implies that  
\begin{align*}
(-1)^{d-1} &\Tr\big(S^*(\lambda) S'(\lambda) (S(\lambda)-\Id)^{d-1}\big)
=\frac{\frac{d}{d\lambda}\Det_d(S(\lambda))}{\Det_d(S(\lambda)}
=\frac{\frac{d}{d\lambda}\Det(S(\lambda))}{\Det(S(\lambda)}+\frac{d}{d\lambda}H_d(\lambda).
\end{align*}
Using Lemma \ref{lem:det-p-to-det} once again, we have that 
\begin{align*}
(-1)^{d-1} &\Tr\big(S^*(\lambda) S'(\lambda) (S(\lambda)-\Id)^{d-1}\big)
=\Tr(S^*(\lambda)S'(\lambda))+\frac{d}{d\lambda}H_d(\lambda)\\
&=\Big[ \Tr(S^*(\lambda)S'(\lambda))-p_d(\lambda)\Big]+\Big[\frac{d}{d\lambda}H_d(\lambda)+p_d(\lambda)\Big].
\end{align*}
By Theorem \ref{thm:Levinson} the integral 
$$
\int_0^\infty \left( \textup{Tr}\left(S(\lambda)^*S'(\lambda)-p_d(\lambda) \right)\right) \, d\lambda
$$ 
is well-defined. 
Since $p_d(\lambda)=P_d'(\lambda)$, Proposition \ref{prop:asymptotic_2} implies that 
\begin{align*}
\int_0^\infty \Big[\frac{d}{d\lambda}H_d(\lambda)+p_d(\lambda)\Big] \, d\lambda=\lim_{\lambda\to\infty} (H_d(\lambda) +P_d(\lambda))-\lim_{\lambda\to 0+}(H_d(\lambda) +P_d(\lambda))\\
&=-H_d(0)-P_d(0). 
\end{align*}

Thus, 
 \begin{align*}
(-1)^{d-1}\frac{1}{2\pi i}&\int_0^\infty \Tr\big(S^*(\lambda) S'(\lambda) (S(\lambda)-\Id)^{d-1}\big)\, d\lambda\\
&=\frac1{2\pi i} \int_0^\infty \left( \textup{Tr}\left(S(\lambda)^*S'(\lambda)-p_d(\lambda) \right)\right) \, d\lambda-\frac{1}{2\pi i} H_d(0)-\frac{1}{2\pi i} P_d(0),
\end{align*}
proving that the integral converges.
\end{proof}

We now prove the main result of this section, which gives reformulation of Theorem \ref{thm:Levinson} as a spectral flow formula. For the cases of dimension 1 and 3, we use the convention of Remark \ref{rmk:path-convention} to fix the path from ${\rm Id}$ to $S(0)$ when it is needed.

\begin{thm}\label{thm:sf_Levinson}
Suppose that $q_1, q_2 \in C_c^\infty(\R^d)$ and $V = q_1q_2$. Then
\[
\textup{sf}(S(\bullet))=\begin{cases}
\frac1{2\pi i} \int_0^\infty \left( \textup{Tr}\left(S(\lambda)^*S'(\lambda)\right)\right) \, d\lambda+\frac{1}{2}, & d=1 \mbox{ and there are no resonances},\\
\frac1{2\pi i} \int_0^\infty \left( \textup{Tr}\left(S(\lambda)^*S'(\lambda)-p_3(\lambda) \right)\right) \, d\lambda+\frac12\Tr(P_s),& d=3, \\
\frac1{2\pi i} \int_0^\infty \left( \textup{Tr}\left(S(\lambda)^*S'(\lambda)-p_d(\lambda) \right)\right) \, d\lambda-\frac{1}{2\pi i} P_d(0),&\text{otherwise}.\end{cases}
\]
\end{thm}
\begin{proof}
Assume firstly that $S(0)=\Id$. Then $S(\bullet)$ is a $C^1$-loop of unitaries in  $\mathcal{U}_d$. Furthermore, by Lemma \ref{lem_integral_comp} we have that the integral $(-1)^{d-1}\frac{1}{2\pi i}\int_0^\infty \Tr\big(S^*(\lambda) S'(\lambda) (S(\lambda)-\Id)^{d-1}\big)\, d\lambda$ converges. Thus, by Theorem \ref{thm:spec-flow-formula-det} and Lemma \ref{lem_integral_comp} we have that 
\begin{align*}
\spf&(S(\bullet))=(-1)^{d-1}\frac{1}{2\pi i}\int_0^\infty \Tr\big(S^*(\lambda) S'(\lambda) (S(\lambda)-\Id)^{d-1}\big)\, d\lambda\\
&=\frac1{2\pi i} \int_0^\infty \left( \textup{Tr}\left(S(\lambda)^*S'(\lambda)-p_d(\lambda) \right)\right) \, d\lambda-\frac{1}{2\pi i} \sum_{l=1}^{d-1}  \frac{(-1)^\ell}{\ell}
\Tr\big((S(0)-\Id)^\ell\big)-\frac{1}{2\pi i} P_d(0).
\end{align*}
Since $S(0)=\Id$, we have that the second term on the right-hand side above is zero, which proves the assertion in the case $S(0)=\Id$. 

Assume now $S(0)\neq \Id$. By Theorem \ref{thm:scat-mat-zero} this can happen only if $d=1$ or $d=3$.

Assume first that $d=1$. In this case $S(\lambda)$ is a $C^1$-path in $\mathcal{U}_1$. Consider firstly the case, when there are no resonances. In this case $S(0)=\begin{pmatrix}
0&-1\\-1&0
\end{pmatrix}$ and $S(0)=e^{i\pi Q}$ with $Q=\frac12\begin{pmatrix}
1&1\\1&1
\end{pmatrix}.$ Concatenating with the path $U_t=e^{i\pi tQ}$, we have that 
\begin{align*}
\spf(S(\bullet))&=\frac1{2\pi i}\int_0^\infty \Tr(S^*(\lambda)S'(\lambda))\, d\lambda+\frac1{2\pi i} \int_0^1 \Tr (U_t^*\dot{U}_t)\, dt\\
&=\frac1{2\pi i}\int_0^\infty \tr(S^*(\lambda)S'(\lambda))\, d\lambda+\frac1{2\pi i} \int_0^1 \Tr (i\pi Q) dt= \frac1{2\pi i}\int_0^\infty \tr(S^*(\lambda)S'(\lambda))\, d\lambda+\frac12,
\end{align*}
as required. In the case when there are resonances, we have that 
$S(0)=\begin{pmatrix} 2 c_+ c_- & c_+^2-c_-^2 \\ c_-^2 - c_+^2 & 2 c_+ c_- \end{pmatrix}$ for some $c_+, c_- \in \R \setminus\{0\}$ with $c_+^2+c_-^2 = 1$. Introduce the path $W_t=e^{i\pi tQ_r}$, where 
$$Q_r=\theta \begin{pmatrix} 0&1\\-1&0\end{pmatrix}, \quad \cos\theta=2c_+c_-,\quad   \sin\theta=c_+^2-c_-^2,$$
so that $e^{i\pi Q_r}=S(0)$. Then 
 \begin{align*}
\spf(S(\bullet))&=\frac1{2\pi i}\int_0^\infty \tr(S^*(\lambda)S'(\lambda))\, d\lambda+\frac1{2\pi i} \int_0^1 \tr (W_t^*\dot W_t)\, dt\\
&=\frac1{2\pi i}\int_0^\infty \tr(S^*(\lambda)S'(\lambda))\, d\lambda+\frac1{2\pi i} \int_0^1 \Tr (i\pi Q_r) dt= \frac1{2\pi i}\int_0^\infty \tr(S^*(\lambda)S'(\lambda))\, d\lambda,
\end{align*}
which completes the proof for $d=1$. 

Assume now that $d=3$. In this case, $S(0)=\Id-2P_s$. Introduce the path $V_t=e^{i\pi t P_s}, t\in[0,1]$. Then the concatenated path $V_\bullet\star S(\bullet)$ is a $C^1$-loop in $\mathcal{U_1}$. Therefore, by Proposition \ref{prop:path-sf-formula} we have that 
$$
\spf(S(\bullet))=\frac1{2\pi i}\int_0^\infty \frac{\frac{d}{d\lambda}\Det(S_3(\lambda)}{\Det_3(S(\lambda))} \, d\lambda+\frac1{2\pi i}\int_0^1 \frac{\frac{d}{dt}\Det_3(V_t)}{\Det_3 V_t}\, dt.
$$
Referring once again to Lemma \ref{lem:det-p-to-det} we have that the first integral can be written as 
\begin{align*}
\frac1{2\pi i}&\int_0^\infty \frac{\frac{d}{d\lambda}\Det(S_3(\lambda))}{\Det_3(S(\lambda))} \, d\lambda\\
&=\frac1{2\pi i} \int_0^\infty \left( \textup{Tr}\left(S(\lambda)^*S'(\lambda)-p_3(\lambda) \right)\right) \, d\lambda-\frac{1}{2\pi i} \sum_{l=1}^{2}  \frac{(-1)^\ell}{\ell}
\Tr\big((S(0)-\Id)^\ell\big)-\frac{1}{2\pi i} P_3(0).
\end{align*}

To compute $\frac1{2\pi i}\int_0^1 \frac{\frac{d}{dt}\Det_3(V_t)}{\Det_3 V_t}\, dt$, we refer to Lemma \ref{lem:logdet_p_lodget} to obtain 
\begin{align*}
\frac1{2\pi i}\int_0^1 \frac{\frac{d}{dt}\Det_3(V_t)}{\Det_3 V_t}\, dt=\frac1{2\pi i} \int_0^1 \Big(\frac{\frac{d}{dt}\Det(V_t)}{\Det V_t}+ \frac{d}{dt} \sum_{l=1}^2 \frac{(-1)^l}{l} \Tr((V_t-\Id)^l)\Big)\, dt.
\end{align*}
Since \(P_s\) is a rank one projection, we have
\[
\Det(V_t)=e^{i\pi t\Tr(P_s)}=e^{i\pi t}.
\]
Therefore, 
$$
\frac{\frac{d}{dt}\Det(V_t)}{\Det V_t}=i\pi 
\qquad
\mbox{so that} 
 \qquad
 \frac1{2\pi i} \int_0^1 \frac{\frac{d}{dt}\Det(V_t)}{\Det V_t}\, dt= \frac12.
 $$

Thus, we have that 
\begin{align*}
&\spf(S(\bullet))\\
&=\Big[\frac1{2\pi i} \int_0^\infty \left( \textup{Tr}\left(S(\lambda)^*S'(\lambda)-p_3(\lambda) \right)\right) \, d\lambda-\frac{1}{2\pi i} \sum_{l=1}^{2}  \frac{(-1)^\ell}{\ell}
\Tr\big((S(0)-\Id)^\ell\big)-\frac{1}{2\pi i} P_3(0)\Big]\\
&\quad +\Big[\frac{1}{2}\Tr(P_s)+  \sum_{l=1}^2 \frac{(-1)^l}{l} \Tr((V_1-\Id)^l)]\\
&=\frac1{2\pi i} \int_0^\infty \left( \textup{Tr}\left(S(\lambda)^*S'(\lambda)-p_3(\lambda) \right)\right) \, d\lambda+\frac12,
\end{align*}
where in the last equality we used the fact that 
$V_1=S(0)$ and that $P_3(0)=0$. Noting that $P_s$ is rank one projection, we conclude the proof.  
%
\end{proof}

\begin{thm}
\label{thm:sf_Levinson2}
Suppose that $q_1, q_2 \in C_c^\infty(\R^d)$ and $V = q_1q_2$.  For $\lambda \geq 0$ let $S(\lambda)$ be the scattering matrix at energy $\lambda$. Then employing the convention of Remark \ref{rmk:path-convention} in dimensions 1 and 3, we have
\begin{align}
\spf(S(\bullet))=\begin{cases} -N-N_{res}, &d=2,4,\\
-N, & \text{otherwise,}\end{cases}
\end{align}
where $N$ is the number of eigenvalues of $H = H_0+V$ and $N_{res}$ is as defined in Theorem \ref{thm:Levinson}.
\end{thm}

\section{Spectral flow for paths of non-densely defined operators}
\label{sec:fctns}

In this section we show that translating from unitaries in $\mathcal{U}_p$ using the Cayley transform allows us to meaningfully discuss spectral flow of paths of non-densely defined operators, or equivalently, paths of self-adjoint operators on varying Hilbert spaces. We then use these tools to describe the unbounded operators on fields of Hilbert spaces defining classes in $KK^1(\C,C_0(\R))$.

Kasparov's stabilisation theorem shows that Hilbert modules of the form $X_{C_0(\R)}$ are (sections of) continuous fields of Hilbert spaces $\{\H_t\}$ over $\R$. Using the stabilisation theorem to make an identification of $X$ with such a continuous field, we provide a characterisation of the families of operators on $\{\H_t\}$ which define Kasparov modules. The same comments apply to any locally compact Hausdorff space in place of $\R$.

\subsection{The Cayley transform of $\mathcal{U}_p$}
\label{sub:the-problem}

It is well-known that the Cayley transform allows to translate between unitaries and (unbounded) self-adjoint operators. 
A careful comparison of spectral flow on unbounded self-adjoint Fredholm operators in different topologies appears in \cite{BLP}, in part using the spectral flow for unitaries and the Cayley transform. Importantly, the class of unitaries considered in \cite{BLP} includes only unitaries $U$ with $U-\Id$ injective. Under this assumption the Cayley transform translates to self-adjoint operators which are densely defined on a fixed Hilbert space $H$.  In \cite{BLP}, the Cayley transform is shown to be continuous between the gap topology and norm topology, while the bounded transform $x\to x(1+x^2)^{-1/2}$ is continuous, by definition, between the Riesz topology and norm topology. 

In this subsection, we aim to show what class of unbounded operators is associated via Cayley transform with the unitaries in $\mathcal{U}_p$ without the assumption that $U-\Id$ is injective. This will allow to define spectral flow of paths of non-densely defined operators, or equivalently, paths of self-adjoint operators on varying Hilbert spaces.


One obvious application encompassed by our results, which we do not explore here, is to spectral flow for boundary value problems, as discussed in \cite{BLP,Wahl07}. In these cases the Hilbert space is fixed, but the domains of the operators are changing in a gap continuous manner.

We start by introducing the notation for the Cayley transform and its inverse. As before, we assume that $\H$ is a separable Hilbert space.

\begin{defn}
\label{defn:cayley-bounded}
Given a self-adjoint unbounded operator $D:\Dom(D)\subset \H\to \H$, the inverse Cayley transform of $D$ is the unitary $C^{-1}(D):=(D+i)(D-i)^{-1}$.

Given a unitary $U\in B(\H)$, the Cayley transform of $U$ is the unbounded self-adjoint operator $C(U):(U-\Id)\H\to \overline{(U-\Id)\H}$ given by $C(U)=i(U+\Id)(U-\Id)^{-1}$.
\end{defn}

\begin{rmk}
We have departed from some standard conventions regarding which is the Cayley transform and which its inverse. 
\end{rmk}

The Cayley transform is ubiquitous. For us, the relevant aspects are the applications of the Cayley transform in studying properties of unbounded self-adjoint operators (see e.g. \cite{Lance}), leveraging the well-understood spectral theory of unitaries.

Next we introduce the class of unbounded operators we shall consider in this section.

\begin{defn}Let $1\leq p<\infty$.
Given a closed subspace $V\subset \H$, we define 
\[
\mathcal{F}_p^{sa}(V)=\{T:\Dom(T)\subset V\to V\ \mbox{self-adjoint}:\,(T-i)^{-1}\in\L^p(V)\}.
\]
Define the set 
\[
\mathcal{F}_p(\H):=\bigcup_{V\subset \H}\mathcal{F}_p^{sa}(V),
\]
where the union is taken over all closed subspaces $V\subset \H$. We define $\mathcal{F}_\infty(\H)$ similarly by assuming that $(T-i)^{-1}$ is a compact operator on the  closed subspace $V$. 

\end{defn}

We can immediately show that the Cayley transform is a bijection between $\mathcal{U}_p$ and $\mathcal{F}_p(\H)$, $1\leq p \leq \infty$. 

\begin{prop}
\label{prop:cayley-biject}
Assume that $1\leq p\leq\infty$. Define the maps 
\begin{align}
C:\mathcal{U}_p\to \mathcal{F}_p(\H),&\quad C(U)=i(U+\Id)(U-\Id)^{-1}\mbox{ densely defined on }\H_U=\overline{(U-\Id)\H},\nonumber\\
C^{-1}:\mathcal{F}_p^{sa}(V)\to \mathcal{U}_p,&\quad C^{-1}(T)=\begin{cases}(T+ i)(T- i)^{-1}&\mbox{ on }V\\ {\rm Id}&\mbox{ on }V^\perp\end{cases}.
\label{eq:inverse-cayley}
\end{align} Then, the maps $C$ and $C^{-1}$ are mutually inverse bijections between $\mathcal{U}_p(\H)$ and $\mathcal{F}_p(\H)$. In particular, $\mathcal{F}_p(\H)$ is a Banach manifold with respect to the metric defined as 
$$d_{ \mathcal{F}_p(\H)}(T,S)=\|(T-i)^{-1}P_V-(S-i)^{-1}P_W\|_{\L_p(\H)}, \quad T\in  \mathcal{F}_p^{sa}(V),\  S\in  \mathcal{F}_p^{sa}(W),$$
where $P_V$ is the orthogonal projection from $\H$ onto the closed subspace $V\subset \H$.
\end{prop}
\begin{proof}
Let $U\in\mathcal{U}_p(\H)$. Then the operator, defined as  
\[
C(U)=i(U+{\rm Id})(U-{\rm Id})^{-1}:(U-{\rm Id})\H\to \H_U
\]
is densely-defined and self-adjoint on $\H_U=\overline{(U-{\rm Id})\H}$, \cite{BLP,BKR}, and therefore, $C$ maps $\mathcal{U}_p$ into $\mathcal{F}_p(\H)$. Clearly, for any $T\in \mathcal{F}_p(\H)$, the operator $C^{-1}(T)$ defined in \eqref{eq:inverse-cayley} is unitary and belongs to $\mathcal{U}_p$.  The choice of domain $\H_U$, as well as identity extension to $V^\perp$, ensures that $C$ and $C^{-1}$ are mutually inverse bijections. The assertion that $ \mathcal{F}_p(\H)$ is a Banach manifold then follows immediately. 
\end{proof}

%
%
%
For $\mathcal{U}_p(\H)$ we have a natural basepoint given by the identity. For $\mathcal{F}_p(\H)$ this point is the zero operator on the zero subspace.
\begin{lemma}
\label{lem:Id-zero}
For $1\leq p\leq\infty$, the Cayley transform maps ${\rm Id}_\H\in\mathcal{U}_p(\H)$ to the zero operator on the zero subspace.
\end{lemma}
\begin{proof}
The domain of the Cayley transform of ${\rm Id}_\H$ is $({\rm Id}_\H-{\rm Id}_\H)\H=\{0\}$, and so the claim follows.  
\end{proof}

Given a path $U_t\in \mathcal{U}_p(\H)$, the operators $C(U_t)$ have not only moving domain, but in general moving {\em Hilbert spaces} on which they are densely-defined. These operators do not admit sensible extensions to the whole Hilbert space $\H$ compatible with the inverse Cayley transform (they would need to be defined as ``$\infty$'' on $\H_U^\perp$) as we can see in \eqref{eq:inverse-cayley}. 

It appears that the useful way to consider extensions of $C(U)$ to all of $\H$ is via graph projection continuity. This approach, described in Proposition \ref{prop:gap-cts-extension} below, extends the characterisation of gap continuity of paths of self-adjoint operators on a fixed space $\H$ in \cite{BLP} to the space $\mathcal{F}_p(\H)$.

\begin{defn}
\label{defn:extend-Cayley-inverse}
Let $1\leq p\leq\infty$.
Given $V\subset \H$ a closed subspace and $T\in \mathcal{F}_p^{sa}(V)$, define
the projection $P_T\in \B(\H)$ by
\[
P_T=P_{G(T)}+0_\H\oplus1_{V^\perp}
\]
where $P_{G(T)}\in B(\H\oplus \H)$ is the graph projection of $T$ extended from $V\oplus V$ to $\H\oplus \H$ by zero, and $1_{V^\perp}$ is the projection $\H\to V^\perp$.
\end{defn}

\begin{rmk}
The disturbing feature of Definition \ref{defn:extend-Cayley-inverse} is that $0\oplus 1_{V^\perp}$ is not the graph of an operator on $V^\perp$. This corresponds to the need to define any extension of $C^{-1}(T_V)$ to $\H$ as ``$\infty$'' on the complement, or equivalently extending the resolvent by $0$.
\end{rmk}

The following lemma show that the topology on $\mathcal{F}_p(\H)$ inherited from $\mathcal{U}_p(\H)$ can be equivalently described as continuity of the projections $P_T$ in the respective Schatten class.

\begin{prop}
\label{prop:gap-cts-extension}
Let $1\leq p\leq\infty$. A path $t\mapsto T_t\in \mathcal{F}_p(\H)$ is continuous  in $\mathcal{F}_p(\H)$ if and only if the path $t\mapsto P_{T_t}$ is continuous in $\L^p(\H\oplus \H)$. In particular, $U_t$ is a continuous path in $\mathcal{U}_p(\H)$ if and only if  $P_{C(U_t)}$ is a  continuous path in $\L^p(\H\oplus \H)$.  In this case, the path $t\mapsto (T_t\pm i)^{-1}$ extended to an operator on $\H$ by zero on $\H_{T_t}^\perp$  is continuous in $\L^p(\H).$
\end{prop}
\begin{proof}Let $t\mapsto T_t\in \mathcal{F}_p(\H)$ be a  continuous path  in $\mathcal{F}_p(\H)$.
 
By definition of the topology on $\mathcal{F}_p(\H)$, the convergence  \(T_t\to T_{t_0}\) in \(\mathcal F_p(\H)\) means that
\[
(T_t-i)^{-1}P_{\H_{T_t}}\to (T_{t_0}-i)^{-1}P_{\H_{T_0}}
\qquad \text{in } \mathcal L^p(\H)
\]
as \(t\to t_0\).
Set
\[
A_t:=(\Id+T_t^2)^{-1}P_{\H_{T_t}}, \qquad B_t:=T_t(\Id+T_t^2)^{-1}P_{\H_{T_t}},
\]
and similarly
\[
A_0:=(\Id+T_{t_0}^2)^{-1}P_{\H_{T_{t_0}}}, \qquad B_0:=T_{t_0}(\Id+T_{t_0}^2)^{-1}P_{\H_{T_{t_0}}}.
\]
The projection $P_{T_t}$ is given by 
\begin{align*}
P_{T_t}&=\begin{pmatrix} (\Id+T_t^2)^{-1}P_{\H_{T_t}}&T_t(\Id+T^2_t)^{-1}P_{\H_{T_t}}\\ T_t(\Id+T^2_t)^{-1}P_{\H_{T_t}}& \Id_{\H_t}-(\Id+T^2_t)^{-1}P_{\H_{T_t}}\end{pmatrix}+0\oplus\Id_{\H^\perp_t}\\
&=\begin{pmatrix} (\Id+T_t^2)^{-1}P_{\H_{T_t}}&T_t(\Id+T^2_t)^{-1}P_{\H_{T_t}}\\ T_t(\Id+T^2_t)^{-1}P_{\H_{T_t}}& \Id_{\H}-(\Id+T^2_t)^{-1}P_{\H_{T_t}}\end{pmatrix}\\
&=
\begin{pmatrix}
A_t & B_t\\
B_t &\Id_\H-A_t
\end{pmatrix},\end{align*}
and similarly, 
\[
P_{T_{t_0}}=
\begin{pmatrix}
A_0 & B_0\\
B_0 & \Id_\H-A_0
\end{pmatrix}.
\]

Since
\[
(T_t-i)^{-1}=(T_t+i)(\Id+T_t^2)^{-1},
\qquad
(T_{t_0}-i)^{-1}=(T_{t_0}+i)(\Id+T_{t_0}^2)^{-1},
\]
we have
\begin{align}\label{eq_T_P_via_graph}
(T_t-i)^{-1}P_{\H_{T_t}}=B_t+iA_t,
\qquad
(T_{t_0}-i)^{-1}P_{\H_{T_{t_0}}}=B_0+iA_0.
\end{align}
Furthermore,
\[
A_t=\frac{(T_t-i)^{-1}P_{\H_{T_t}}-\big((T_t-i)^{-1}P_{\H_{T_t}}\big)^*}{2i},
\qquad
B_t=\frac{(T_t-i)^{-1}P_{\H_{T_t}}+\big((T_t-i)^{-1}P_{\H_{T_t}}\big)^*}{2},
\]
and, similarly, for \(A_0,B_0\). Hence \(T_t\to T_{t_0}\) in \(\mathcal F_p(\H)\) implies that
\[
A_t\to A_0,\qquad B_t\to B_0
\]
in \(\mathcal L^p(\H)\) and so
\begin{equation}\label{eq_P_T_via_A,B}
P_{T_t}-P_{T_{t_0}}=
\begin{pmatrix}
A_t-A_0 & B_t-B_0\\
B_t-B_0 & -(A_t-A_0)
\end{pmatrix},
\end{equation}
ensuring that \(P_{T_t}\to P_{T_{t_0}}\) in \(\mathcal L^p(\H\oplus\H)\).

Conversely, assume that
\[
P_{T_t}\to P_{T_{t_0}} \qquad \text{in } \mathcal L^p(\H\oplus\H)
\]
as \(t\to t_0\). Then, by Equation \eqref{eq_P_T_via_A,B} we have that  \(
A_t\to A_0,\  B_t\to B_0\) as $t\to t_0$. Equality \eqref{eq_T_P_via_graph} implies that 
\[
(T_t-i)^{-1}P_{\H_{T_t}}-(T_{t_0}-i)^{-1}P_{\H_{T_{t_0}}}
=(B_t-B_0)+i(A_t-A_0)\to 0
\]
in \(\mathcal L^p(\H)\), proving that \(T_t\to T_{t_0}\) in \(\mathcal F_p(\H)\), as required. 
%
%
%
%
\end{proof}


\begin{corl}
\label{cor:strong-cts}
Let $t\mapsto T_t$ be a continuous path in $\mathcal{F}_\infty(\H)$ with $\H_t=\overline{\Dom(T_t)}\subset \H$. Then the following are equivalent:
\begin{enumerate}[noitemsep]
\item  The graph projections $P_{G(T_t)}$ are strongly continuous.
\item The projections onto $\H_t$ are strongly continuous.
\end{enumerate}

\end{corl}
\begin{proof}The assertion follow immediately from the norm continuity of  
$P_{T_t}\oplus\big(\begin{smallmatrix}0&0\\0&1_{\H_t^\perp}\end{smallmatrix}\big)$.
\end{proof}

In the following lemma we specify how continuity of a unitary path translates to the continuity of the associated  path of self-adjoint operators under an additional assumption that they are all densely defined on a fixed Hilbert space.

\begin{lemma}
\label{lem:norm-res-cts}
Let $t\mapsto U_t$ be a continuous $p$-Schatten path of unitaries. Then if $C(U_t)$ are all defined on a fixed $\H$, the map $t\mapsto C(U_t)$ is $p$-norm resolvent continuous on $\H$, that is $\|(C(U_t)-\lambda)^{-1}-(C(U_s)-\lambda)^{-1}\|_p\to 0$ as $t\to s$ for all $\lambda\notin{\rm spec}(C(U_s))$.
\end{lemma}
\begin{proof}
We will prove the result for $\lambda\not\in\R$, which ensures that $\lambda\not\in{\rm spec}(C(U_t))$. The case of real $\lambda\not\in{\rm spec}(C(U_t))$ follows similarly using \cite[Section VIII.7]{RS1}. Normality of $U_t$ and the resolvent identity yields
\begin{align*}
&(C(U_t)-\lambda)^{-1}-(C(U_s)-\lambda)^{-1}\\
&=(U_t-1)(i(U_t+1)-\lambda(U_t-1))^{-1}-(U_s-1)(i(U_s+1)-\lambda(U_s-1))^{-1}\\
&=(i(U_t+1)-\lambda(U_t-1))^{-1}(U_t-1)-(U_s-1)(i(U_s+1)-\lambda(U_s-1))^{-1}\\
&=(i(U_t+1)-\lambda(U_t-1))^{-1}\Big((U_t-1)(i(U_s+1)-\lambda(U_s-1))\\
&\qquad\qquad\qquad\qquad-(i(U_t+1)-\lambda(U_t-1))(U_s-1)\Big)(i(U_s+1)-\lambda(U_s-1))^{-1}\\
&=(i(U_t+1)-\lambda(U_t-1))^{-1}\Big((i+\lambda)(U_t-U_s)\Big)(i(U_s+1)-\lambda(U_s-1))^{-1}\\
&=((i-\lambda)U_t+(i+\lambda))^{-1}(i+\lambda)(U_t-U_s)((i-\lambda)U_s+(i+\lambda))^{-1}.
\end{align*}
Now as $\lambda\not\in\R$ we have $(\lambda+i)(\lambda-i)^{-1}\not\in\mathbb{T}$ and the spectral mapping property then gives (for $\lambda\neq i$)
\begin{align*}
\Vert (C(U_t)&-\lambda)^{-1}-(C(U_s)-\lambda)^{-1}\Vert_p\\
&\leq \Big|\frac{i+\lambda}{(i-\lambda)^2}\Big|\Vert U_t-U_s\Vert_p \,d({\rm spec}(U_t),(\lambda+i)(\lambda-i)^{-1})\,d({\rm spec}(U_s),(\lambda+i)(\lambda-i)^{-1}),
\end{align*}
where $d$ is the usual distance in $\C$. For $\lambda=i$ we find
\[
\Vert (C(U_t)-\lambda)^{-1}-(C(U_s)-\lambda)^{-1}\Vert_p\leq \frac{1}{2}\Vert U_t-U_s\Vert_p.\qedhere
\]
\end{proof}
Every bounded continuous function of $C(U_t)$ is strongly continuous and every $C_0$ function of $C(U_t)$ is norm continuous by the same proof as \cite[VIII.20]{RS1}. 
In particular, we obtain the following corollary. 
\begin{corl}
\label{cor:continuity-props}
Let $t\mapsto U_t\in\mathcal{U}_\infty$ be a uniform norm continuous path of unitaries such that the Cayley transforms $C(U_t)$ are all defined on a fixed Hilbert space $\H$.
Then the path $t\mapsto C(U_t)(1+C(U_t)^2)^{-1/2}$ is strongly continuous and for all $\epsilon>0$, the path $t\mapsto C(U_t)(1+C(U_t)^2)^{-1/2-\epsilon}$ is norm continuous.
\end{corl}

Having an equivalent description of continuity of a path $t\mapsto T_t$ in $\mathcal{F}_\infty(\H)$ we can now define spectral flow for a path of operators which are not densely defined in the same Hilbert space.

\begin{defn}
\label{defn:SF_var_domain}Let $1\leq p\leq \infty$ and $t\mapsto T_t$, $t\in[0,1]$, be a continuous path in $\mathcal{F}_\infty(\H)$. Define the \emph{spectral flow} for $T_\bullet$ by setting 
\[
\textup{sf}(T_\bullet)=\textup{sf}(C^{-1}(T_\bullet)).
\]
\end{defn}

Translating from unitaries, we obtain exact forms on $\mathcal{F}_p(\H)$ and integral formulae for spectral flow.

\begin{defn}
\label{defn:one-forms_Cayley}
For $x\geq 0$ let 
\[
C_x^{-1}=\frac{\sqrt{\pi}\Gamma(x+1/2)}{\Gamma(x+1)}.
\]
Let $n$ be an integer $n\geq p-1$ and $r$ real with $r\geq (p-1)/2$. For $U\in\mathcal{U}_p(\H)$ set $\H_U=\overline{(U-{\rm Id})\H}$. Define one-forms on the 
tangent space to $\mathcal{F}_p(\H)$  at the point $C(U)$ by 
\[
\alpha_{C(U),n}(X):=C_{n/2}\frac{1}{2i}\Tr_{\H_U}(X(C(U)-i)^{-n})=C_{n/2}\Big(\frac{1}{2i}\Big)^{n+1}\Tr_\H(X(U-{\rm Id})^n)
\]
and for real $r\geq (p-1)/2$ define
\[
\beta_{C(U),r}(X):=-C_r\frac{1}{2}\Tr_{\H_U}(X|C(U)-i|^{-2r})=-iC_r\Big(\frac{1}{2}\Big)^{2r+1}\Tr_\H(X|U-{\rm Id}|^{2r}).
\]
\end{defn}
\begin{thm}
\label{thm:exact-cayley}
For $p\geq1$ and $n\geq p-1$ integral and $r$ real with $r\geq (p-1)/2$, the forms $\alpha_{C(U),n},\beta_{C(U),r}$ are exact.
\end{thm} 

Theorem \ref{thm:exact-cayley} follows either via the Cayley transform or by using the following observations and arguing as in Section \ref{subsec:exact-forms}. Briefly, recall that
the exponential map takes tangent vectors $X\in T_{\rm Id}\mathcal{U}_p$  to right translations, 
\[
\exp(sX)\cdot U=Ue^{sX},\quad U\in \mathcal{U}_p,\quad X=-X^*\in \L^p(\H).
\]
We transfer this action to $\mathcal{F}_p(\H)$ via the Cayley transform $C$ so that
\[
\exp(sX)\cdot C(U):=C(Ue^{sX}).
\]
Since
\[
(C(U)-i)^{-1}=-\frac{i}{2}(U-{\rm Id})=\frac{1}{2i}(U-{\rm Id})
\]
we have
\[
\frac{\d}{\d s}\Big|_{s=0}(C(Ue^{sX})-i)^{-1}
=-\frac{i}{2}\frac{\d}{\d s}\Big|_{s=0}(Ue^{sX})=\frac{1}{2i}UX.
\]
With these observations Theorem \ref{thm:exact-cayley} can be proved directly.

Similarly, via the Cayley transform, Theorem \ref{thm:cayley-spec-flow-formula} provides an integral formula for the spectral flow for a path of operators defined in varying Hilbert spaces. To state this theorem, we firstly define a notion of ``vanishing at infinity'' for a path of operators from $\mathcal{F}_\infty(\H)$. 

\begin{defn}
\label{defn:cvge-zero}Let $1\leq p\leq \infty$. 
We say that a path $\R\ni t\mapsto T_t\in \mathcal{F}_p(\H)$ vanishes at infinity (in Schatten $p$-norm) if $P_{T_t}\to \begin{pmatrix}
0&0\\0&\Id_\H
\end{pmatrix}$ in norm $\L^p(\H\oplus \H)$ (or equivalently, $T_t$ converges to the zero operator on the zero subspace in $p$-norm resolvent sense). 
\end{defn}

Proposition \ref{prop:gap-cts-extension} implies that  if the path $T_t$ vanishes at infinity in the Schatten $p$-norm then the path of unitaries
\[
C^{-1}(T_t)=\begin{cases}(T_t+ i)(T_t- i)^{-1}&\mbox{ on }\H_t\\ {\rm Id}&\mbox{ on }\H_t^\perp\end{cases}
\]
converges to ${\rm Id}_\H$ in the Schatten $p$-norm, and so corresponds to a loop.

\begin{thm}
Let $p\geq 1$ and let $\R\ni t\mapsto T_t$ be a differentiable path in $\mathcal{F}_p(\H)$ which vanishes at infinity in Schatten $p$-norm. 
Then for $r\geq (p-1)/2$ or $n\geq p-1$ we have 
\begin{align*}
\textup{sf}(T_\bullet)&=\frac{-i}{2^{2r+1}}\frac{\Gamma(r+1)}{\sqrt{\pi}\Gamma(r+1/2)}\int_0^1\Tr\Big(C^{-1}(T_t)^*\frac{\d}{\d t}(C^{-1}(T_t))|C^{-1}(T_t)-\Id|^{2r}\Big)\,\d t\\
&=(-1)^n\frac{1}{2\pi i}\int_0^1\Tr\Big(C^{-1}(T_t)^*\frac{\d}{\d t}(C^{-1}(T_t))(C^{-1}(T_t)-\Id)^{n}\Big)\,\d t,
\end{align*}
whenever the integral converges. 
\end{thm}

To complete this subsection, we discuss the application of our results for continuous fields of Hilbert spaces.  With the description of continuity of paths in $\mathcal F_\infty(H)$ we immediately obtain the following characterisation.
\begin{thm}
Let $(-1,1)\ni t\mapsto H_t\subset H$ be a continuous field of Hilbert spaces, and $T_t$ a family of densely-defined self-adjoint operators on each $H_t$.

Then the set
\[
\Gamma=\{\xi_\bullet:\,\xi_t\in{\rm Dom}(T_t)\ \mbox{and}\  t\mapsto\langle\xi_t,\xi_t\rangle_t\ \mbox{ is continuous}\} 
\] 
is dense in $\int_{(-1,1)}^\oplus H_t$ if and only if
$t\mapsto T_t$ is a continuous path in $\mathcal{F}_\infty(H)$ and the graph projections $t\mapsto P_{G(T_t)}$ are strongly continuous.
Moreover, such a path has $t\mapsto (T_t\pm i)^{-1}$ in $C_0(K(H))$ provided $(T_t)$ vanishes at infinity.
\end{thm}

\subsection{Kasparov modules}

We now relate our  constructions to Kasparov theory.  Namely, we  will show that unbounded Kasparov modules $(\C,X_{C_0(\R)},T)$ can be described precisely using paths of operators $T_t\in  \mathcal{F}_\infty(\H)$  which vanish at infinity in the operator norm. 

Recall that an unbounded Kasparov module $(\C,X_{C_0(\R)},T)$ representing a class in $KK^1(\C,C_0(\R))$ consists of a right Hilbert module $X$ over $C_0(\R)$ and $T:\Dom(T)\subset X\to X$ a self-adjoint regular operator whose resolvent $(T\pm i)^{-1}$ is a compact endomorphism of $X$. All of these notions can be found in \cite{Lance}.

For a complex $C^*$-algebra $A$, \cite[Section 3]{BKR} shows that the Cayley transform gives an explicit isomorphism $K_1(A)\to KK^1(\C,A)$.
In the case $A=C_0(\R)$, elements of $K_1(A)$ are represented by elements of $\mathcal{U}_\infty$. Using the identification of $\mathcal{U}_\infty$ and $\mathcal{F}_\infty$ provides most of the information we require.


\begin{thm}
\label{prop:KK-paths}
Every unbounded Kasparov module $(\C,X_{C_0(\R)},T)$ defining a class in $KK^1(\C,C_0(\R))$ is unitarily equivalent to a cycle $(\C,P_\bullet(\H\ox C_0(\R)),\tilde{T}_\bullet)$ with $t\mapsto P_t$ a strongly continuous path of projections on $\H$ and $t\mapsto \tilde{T}_t$ a continuous path in $\mathcal{F}_\infty(\H)$ vanishing at infinity.

Conversely, every continuous path $t\mapsto T_t$ in $\mathcal{F}_\infty(\H)$  vanishing at infinity  whose graph projections $P_{G(T_t)}$ are strongly continuous defines an unbounded Kasparov module $(\C,(\H_{\bullet})_{C_0(\R)},T_\bullet)$ and so a class in $KK^1(\C,C_0(\R))$. Here
$\H_t\subset \H$ is the subspace on which $T_t$ is densely-defined.
\end{thm}
\begin{proof}
Given an unbounded Kasparov module $(\C,X_{C_0(\R)},T)$, let $v:X\to \H\ox C_0(\R)$ be an isometry for some separable Hilbert space $\H$: such a $v$ exists by Kasparov's stabilisation theorem \cite{Ka1.5}. Then the unbounded cycle $(\C,X_{C_0(\R)},T)$ is unitarily equivalent, via $v$, to
\[
(\C,vv^*(\H\ox C_0(\R)),vTv^*).
\]
The map
\[
\begin{pmatrix} 0&v^*\\v&0\end{pmatrix}:\begin{pmatrix} X\\ \H\ox C_0(\R)\end{pmatrix} \to \begin{pmatrix} X\\ \H\ox C_0(\R)\end{pmatrix}
\]
is a bounded self-adjoint operator, and so too then is its square,
\[
\begin{pmatrix} {\rm Id}_X&0\\0&vv^*\end{pmatrix}
\]
as is the composition with  the evaluation at $t$, denoted $ev_t$. For the second component $vv^*$ we recall from \cite[Section 2.3]{Raeburn-Williams} that
\[
\End_{C_0(\R)}(\H\ox C_0(\R))\cong C_b(\R,\B(\H)_{*-s})
\]
where $\B(\H)_{*-s}$ is the bounded operators with the $*$-strong topology. Hence $t\mapsto ev_t\circ vv^*$ is a strongly continuous family of projections.
Let $v_t:X\to \H$ be the composition $ev_t\circ v$.

The compactness of the resolvent of $T$ gives the compactness of $(\pm ivv^*+vTv^*)^{-1}$. Using the results of \cite[Section 2.3]{Raeburn-Williams}
\[
\End_{C_0(\R)}^0(\H\ox C_0(\R))\cong C_0(\R,\K(\H))
\]
we see that $t\mapsto (\pm iv_tv_t^*+v_tTv_t^*)^{-1}$ is a norm continuous family of compacts vanishing at infinity. As the domain of $v_tTv_t^*$ is the image of $(\pm iv_tv_t^*+v_tTv_t^*)^{-1}$, all paths of vectors in the domain of $v_tTv_t^*$ converge to zero as $t\to\pm\infty$. Thus $v_tTv_t^*$ vanishes at infinity.

Moreover, the family of graph projections 
\[
P_{G(v_tTv_t^*)}=\begin{pmatrix} (v_tv_t^*+(v_tTv_t^*)^2)^{-1} & v_tTv_t^*(v_tv_t^*+(v_tTv_t^*)^2)^{-1}
\\ v_tTv_t^*(v_tv_t^*+(v_tTv_t^*)^2)^{-1} & v_tv_t^*-(v_tv_t^*+(v_tTv_t^*)^2)^{-1}
\end{pmatrix}
\]
is strongly continuous, and $P_{G(v_tTv_t^*)}\oplus\big(\begin{smallmatrix} 0 & 0\\ 0&1-v_tv_t^*\end{smallmatrix}\big)$ is norm continuous. Hence the family 
$t\mapsto v_tTv_t^*$ is in $F_\infty(\H)$, or equivalently, the path of unitaries
\[
t\mapsto \begin{cases} (iv_tv_t^*+v_tTv_t^*)(- iv_tv_t^*+v_tTv_t^*)^{-1} & v_tv_t^*\H\\ (1-v_tv_t^*) & (1-v_tv_t^*)\H\end{cases}
\]
is in $\mathcal{U}_\infty(\H)$.

For the converse, let $t\mapsto T_\bullet\in\mathcal{F}_\infty(\H)$ be a continuous path, and let $\H_t$ be the closure of $\Dom(T_t)$ in $\H$. Let 
\[
X_{C_0(\R)}=\{\sigma:\R\to \H:\,\sigma\in C_0(\R,\H),\,\sigma(t)\in \H_t\}
\]
with the multiplication action of $C_0(\R)$ and inner product 
\[
\pairing{\sigma_1}{\sigma_2}_X(t)=\pairing{\sigma_1(t)}{\sigma_2(t)}_\H. 
\]
Then $X$ is a Hilbert module, in particular complete, by \cite[Example 2.13]{Raeburn-Williams}.

Let $T_\bullet$ be the unbounded operator on $X_{C_0(\R)}$ given by
\[
\Dom(T_\bullet)=\{\sigma\in X:\,\sigma(t)\in \Dom(T_t)\},\qquad (T_\bullet\sigma)(t)=T_t(\sigma(t)).
\]
Then $T_\bullet$ is regular and self-adjoint. To see this, consider the graph projection $P_{G(T_\bullet)}:X\oplus X\to X\oplus X$. By the local-global principle of Pierrot and Kaad-Lesch \cite{KL12,P}, the submodule $P_{G(T_\bullet)}(X\oplus X)$ is complemented if and only if \cite[Proposition 1.6]{MR} for all $t\in\R$
\[
\big((P_{G(T_\bullet)}(X\oplus X))^\perp\big)_t=(P_{G(T_t)}(\H_t\oplus \H_t))^\perp\subset \H\oplus \H.
\]
Given $(\sigma_1,\sigma_2)\in (P_{G(T_\bullet)}(X\oplus X))^\perp$, for every $\rho\in \Dom(T_\bullet)$ we have for every $t\in\R$
\[
0=\pairing{(\rho,T_\bullet\rho)}{(\sigma_1,\sigma_2)}_{X\oplus X}(t)
=\pairing{(\rho(t),T_t\rho(t))}{(\sigma_1(t),\sigma_2(t))}_{\H_t\oplus \H_t}.
\]
Thus $(\sigma_1(t),\sigma_2(t))\in (P_{G(T_t)}(\H_t\oplus \H_t))^\perp$ for each $t\in\R$, and we have the inclusion $\big((P_{G(T_\bullet)}(X\oplus X))^\perp\big)_t\subset(P_{G(T_t)}(\H_t\oplus \H_t))^\perp$.

Conversely, let $(\xi_t,\zeta_t)\in (P_{G(T_t)}(\H_t\oplus \H_t))^\perp$. By assumption, the graph projections are strongly continuous, so choosing any
$(\xi,\zeta)\in (\H\oplus \H)\ox C_0(\R)$ whose evaluation at $t$ is $(\xi_t,\zeta_t)$, we obtain a continuous path $(1-P_{G(T_\bullet)})(\xi,\zeta)$. By construction, 
$(1-P_{G(T_\bullet)})(\xi,\zeta)\in (P_{G(T_\bullet)}(X\oplus X))^\perp$ and $(1-P_{G(T_\bullet)})(\xi,\zeta)_t=(\xi_t,\zeta_t)$. Thus $\big((P_{G(T_\bullet)}(X\oplus X))^\perp\big)_t\supset(P_{G(T_t)}(\H_t\oplus \H_t))^\perp$, and $T_\bullet$ is regular. Self-adjointness now follows from the local-global principle and the pointwise 
self-adjointness of the $T_t$.

Moreover $t\mapsto {\rm Id}_\H-C^{-1}(T_t)=-2i(T_t-i)^{-1}$ converges to zero in norm, and for each $t$ is a compact operator on $\H$. Hence 
$(T_\bullet-i)^{-1}$ is a section of the bundle over $\R$ with fibre at $t\in\R$ the compacts $\K(\H_t)$ on $\H_t$. Since $(T_t-i)^{-1}$ extends continuously to a compact on $\H$ by Lemma \ref{prop:gap-cts-extension}, $(T_\bullet-i)^{-1}$ is contained in $C_0(\R,\K(\H))=\End^0_{C_0(\R)}(\H\ox C_0(\R))$. Thus $(T_\bullet-i)^{-1}$ is a compact endomorphism on $\H\ox C_0(\R)$ commuting with the projection $Q_\bullet$ onto the sections of the Hilbert space bundle $\H_t\to t$. 
Hence $(T_\bullet-i)^{-1}$  is a compact operator on $X$, and we obtain a Kasparov module.
\end{proof}


\begin{thebibliography}{9999}

\bibitem{agmon75}
S. Agmon. 
{\em Spectral properties of {S}chr{\"{o}}dinger operators and scattering theory}, Accademia Nazionale dei Lincei; Scuola Normale Superiore di Pisa, Pisa, 1975.

\bibitem{alexander24}
A. Alexander.
{\em Trace formula and Levinson's theorem in the presence of resonances},
Rev. Math. Phys., {\bf 37} (4), Article ID 2450036, 2025.


\bibitem{ANRR}
A. Alexander, D. T. Nguyen, A. Rennie, S. Richard.
{\em Levinson's theorem for two-dimensional scattering systems: it was a surprise, it is now topological!}, J. Spec. Th., {\bf 14} (2024), 991--1031.

\bibitem{AR23}
A. Alexander, A. Rennie.
{\em Levinson's theorem as an index pairing}, J. Funct. Anal., {\bf 286} (5), 2024.

\bibitem{AR23-4D}
A. Alexander, A. Rennie.
{\em The structure of the wave operator in four dimensions in the presence of resonances}, Lett. Math. Phys., {\bf 114} (2024), 122.



      
\bibitem{aps75}
M. F. Atiyah, V. K. Patodi, I. M. Singer.
{\em Spectral asymmetry and {R}iemannian geometry. {I}},
Math. Proc. Cambridge Philos. Soc., {\bf 77},  45--69 (1975)

\bibitem{aps76}
M. F. Atiyah, V. K. Patodi, I. M. Singer.
{\em Spectral asymmetry and {R}iemannian geometry. {III}},
Math. Proc. Cambridge Philos. Soc., {\bf 79},  71--99 (1976)

\bibitem{ADT} N. Azamov, T. Daniels, Y. Tanaka, {\em A topological approach to unitary spectral flow via continuous enumeration of eigenvalues}, J. Funct. Anal. {\bf 281} (2021) 109152.

\bibitem{BCPRSW}
M. T. Benameur, A. Carey, J. Phillips, A. Rennie, F. Sukochev, K. Wojciechowski.
{\em An analytic  approach to spectral flow in von {N}eumann algebras}, 
in `Analysis, geometry and topology of elliptic operators',  297--352. World Sci. Publ. (2006)

\bibitem{bolle86}
D. Boll\'{e}, F. Gesztesy, C. Danneels, S. F. J. Wilk,
{\em Threshold behaviour and Levinson's theorem for two-dimensional scattering systems: a surprise}, Phys. Rev. Lett. {\bf 56}, 1986, 900--903.

\bibitem{bolle85} D. Boll\'{e}, F. Gesztesy, S. F. J. Wilk. {\em A complete treatment of low-energy scattering in one dimensions}, J. Operator Theory, 1985.

\bibitem{bolle88} D. Boll\'{e}, C. Danneels, F. Gesztesy. {\em Threshold scattering in two dimensions}, Ann. Inst. Henri Poincar\'{e} Phys. Th\'{e}or., {\bf 48} (2), 1988.

\bibitem{bolle77}
D. Boll\'{e}, T. A. Osborn. {\em An extended {L}evinson's theorem}, J. Mathematical Phys., {\bf 18} (3), 1977, 432--440.

\bibitem{BooWoj} B. Boo{\ss}-Bavnbek, K. P.  Wojciechowski, {\em Elliptic boundary problems for {D}irac operators},
    Birkh\"auser, Boston 1993.
    
    \bibitem{BBF} B. Boo{\ss}-Bavnbek, K. Furutani, {\em
The Maslov Index: a Functional Analytical Definition
and the Spectral Flow Formula}, Tokyo J. Math.,
{\bf 21} (1), 1998.


\bibitem{BLP} B. Boo{\ss}-Bavnbek, M. Lesch, J. Phillips,
 {\em Unbounded {F}redholm operators and spectral flow}, Canad. J. Math., {\bf 57} (2),
     2005, 225--250.


\bibitem{BKR} C. Bourne, J. Kellendonk, A. Rennie, {\em The Cayley transform in complex, real and graded $K$-theory}, Int. J. Math. (2020) 2050074.


\bibitem{CGRS2}
 A. Carey, V. Gayral, A. Rennie, F. Sukochev, {\em Index theory for locally compact noncommutative spaces}, 
Mem. Amer. Math. Soc., {\bf 231} (1085), 2014.

\bibitem{CP1}
A. Carey, J. Phillips.
{\em Unbounded {F}redholm modules and spectral flow}, 
Canad. J. Math. {\bf 50} (4),  673--718 (1998)

\bibitem{CPRS1}
A. Carey, J. Phillips, A. Rennie, F. Sukochev. {\em The local index formula in semifinite von {N}eumann algebras {I}: {S}pectral flow}, Adv. Math., {\bf 202} (2), 2006, 451--516.



\bibitem{CPotSuk} A. Carey, D. Potapov, F. Sukochev, {\em Spectral flow is the integral of one forms on the Banach manifold of self adjoint Fredholm operators}, Adv. Math., {\bf 222} (5), 2009, 1809--1849.


\bibitem{dlH} P. de la Harpe, {\em Classical Banach-Lie groups and Banach-Lie algebras of operators in Hilbert space}, PhD Thesis, University of Warwick, 1972.

\bibitem{DSBW} N. Doll, H. Schulz-Baldes, N. Waterstraat, {\em Spectral flow: A functional analytic and index-theoretic approach}, de Gruyter Studies in Matthematics, Vol 94, 2023.

\bibitem{zworski19}
S. Dyatlov, M. Zworski. {\em Mathematical theory of scattering resonances}, volume 200 of
  Graduate Studies in Mathematics. American Mathematical Society, Providence, RI, 2019.



\bibitem{Getzler}
E. Getzler, {\em The odd Chern character in cyclic homology and spectral flow,}  Topology {\bf 32} (3) (1993) 489--507.

\bibitem{GGK} I. Gohberg, S. Goldberg, N. Krupnik, {\em Traces and Determinants of Linear Operators}, Birkh\"{a}user Verlag, 2000.

\bibitem{guillope81}
L. Guillop\'{e}. {\em Une formule de trace pour l'op\'{e}rateur de {S}chr\"{o}dinger}, PhD Thesis, Universit\'{e} Joseph Fourier Grenoble, 1981. 

\bibitem{jensen79}
A. Jensen, T. Kato. {\em Spectral properties of {S}chr\"{o}dinger operators and time-decay of
  the wave functions}, Duke Math. J., {\bf 46} (3), 1979, 583--611.
  


\bibitem{jensen80}
A. Jensen. {\em Spectral properties of {S}chr\"{o}dinger operators and time-decay of the wave functions results in {$L^{2}(\R^{m})$}, {$m\geq 5$}},  Duke Math. J., {\bf 47} (1), 1980, 57--80.

\bibitem{KL12} J. Kaad and M. Lesch, \emph{A local global principle for regular operators in Hilbert $C^{*}$-modules}, J. Funct. Anal., {\bf 262} (10) (2012), 4540-4569.
  
  \bibitem{KL13}
{J.~Kaad} and {M.~Lesch}. \emph{Spectral flow and the unbounded
  {K}asparov product}, Adv. Math. \textbf{248} (2013), 495--530.
  
  \bibitem{Ka1.5} G. G. Kasparov, {\em Hilbert $C^*$-modules: theorems of Stinespring and Voiculescu}, J. Operator Theory, {\bf 4} (1980), 133--150.


%
\bibitem{kellendonk08}
J. Kellendonk, S. Richard.
{\em  On the structure of the wave operators in one dimensional potential
  scattering}, Math. Phys. Electron. J., {\bf 14}, 1-21, 2008.

%
\bibitem{kellendonk12}
J. Kellendonk, S. Richard. {\em On the wave operators and {L}evinson's theorem for potential
  scattering in {$\mathbb{R}^3$}}, Asian-Eur. J. Math., {\bf 5} (1), 2012.

  
\bibitem{KirLes} P. Kirk, M. Lesch,
     {\em The {$\eta$}-invariant, {M}aslov index, and spectral flow for
              {D}irac-type operators on manifolds with boundary},
   Forum Math., {\bf16} (4), 2004, 553--629.
   
   \bibitem{KS80} 
M. Klaus, B. Simon.
{\em Coupling constant thresholds in nonrelativistic quantum mechanics. {I}. Short-range two-body case}, Ann. Physics, {\bf 130} (2), 1980, 251--281.
   
   
\bibitem{Lance} E. C. Lance, {\em Hilbert $C^*$-modules. A toolkit for operator algebraists}. 
London Mathematical Society Lecture Note Series, 210. Cambridge University Press, Cambridge 1995. {\rm x}+130 pp.


     
\bibitem{LU} M. Lesch, {\em  The uniqueness of the spectral flow on spaces of unbounded
              self-adjoint {F}redholm operators}, in
 Spectral geometry of manifolds with boundary and decomposition
              of manifolds,
    Contemp. Math.,
    {\bf 366}, Amer. Math. Soc., Providence, RI, 2005, 193--224.
    
    \bibitem{levinson49}
N. Levinson. {\em On the uniqueness of the potential in a {S}chr\"{o}dinger equation for a given asymptotic phase}, Danske Vid. Selsk. Mat.-Fys. Medd., {\bf 25} (9), 1949.

\bibitem{MR} B. Mesland, A. Rennie, {\em Friedrichs angle and alternating projections in Hilbert $C^*$-modules}, J. Math. Anal. App., {\bf 516} (2022), 126474, 19pp.

%
\bibitem{ped} Pedersen, G. K., {\em $C^*$-algebras and their automorphism groups}, London Mathematical Society Monographs, 14. Academic Press, Inc., 1979, {\rm ix} + 416 pp.

\bibitem{Phi96}
    J. Phillips, 
     {\em Self-adjoint Fredholm operators and spectral flow}, Canad. Math. Bull., {\bf 39} (4),
      1996,
    460--467.
    

\bibitem{Phi97}
J. Phillips.
{\em Spectral flow in type {I} and {II} factors --- a new approach}, 
in `Cyclic cohomology and noncommutative geometry', Fields Inst. Commun., {\bf 17},  137--153 (1997)

\bibitem{P} F. Pierrot, {\em Op\'erateurs r\'eguliers dans les $C^{*}$-modules
et structure des $C^{*}$-alg\'ebres de groupes de Lie semisimples complexes simplement connexes}, J. Lie theory, {\bf 16} (2006), 651--689.

    
\bibitem{Pu} A. Pushnitski, {\em The spectral shift function and the invariance principle}, J. Funct. Anal., {\bf 183}, 2001, 269--320.

\bibitem{Raeburn-Williams} I.~Raeburn and D.~Williams,
{\em Morita equivalence and continuous-trace $C^*$-algebras},
Mathematical Surveys and Monographs, 60. 
American Mathematical Society, Providence, RI, 1998. xiv+327 pp. 
ISBN: 0--8218--0860--5.
    
\bibitem{RS1} M. Reed, B. Simon, {\em Methods of Modern Mathematical Physics I: Functional Analysis}.
Academic Press Inc., 1972.

%

\bibitem{richard21}
S. Richard, R. Tiedra de Aldecoa, L. Zhang.
{\em Scattering operator and wave operators for 2{D} {S}chr\"{o}dinger operators with threshold obstructions}, Complex Anal. Oper. Theory, {\bf 15} (6) 2021.




\bibitem{Sakai77} T. Sakai, {\em On cut loci of compact symmetric spaces}, Hokkaido Math. J. {\bf 6} (1977), 136--161.

\bibitem{Simon77} B. Simon, {\em Notes on infinite determinants of Hilbert space operators}, Adv. Math., {\bf 24} (1977), 244--273.

\bibitem{Simon05} B. Simon, {\em Trace ideals and their applications}, AMS, 2nd edition, 2005.

\bibitem{Wahl07} 
C. Wahl C. \emph{On the noncommutative spectral flow}. {J.~Ramanujan Math. Soc.}, 
 \textbf{22} (2007), 135--187. 

\bibitem{WahlMems} C. Wahl, {\em Noncommutative Maslov index and eta-forms}, Mems. Am. Math. Soc, {\bf 189} (2007).

\bibitem{Wahl08}
C. Wahl. \emph{A new topology on the space of unbounded selfadjoint operators, $K$-theory and spectral flow}. In $C^\ast$-algebras and elliptic theory II, 297--309, Trends Math., Birkhäuser, Basel (2008).

\bibitem{Wahl09} C. Wahl, {\em Homological index formulas for elliptic operators over $C^*$-algebras}, New York J. Math., {\bf 15}, 319--351 (2009).



\bibitem{yafaev10}
D. R. Yafaev. {\em Mathematical scattering theory: Analytic theory}, {\bf 158}, 
  Mathematical surveys and monographs, American Mathematical Society, 2010.


\end{thebibliography}
\end{document}